\documentclass[11pt]{amsart}
\usepackage[T1]{fontenc}
\usepackage{lmodern,microtype,amsmath,amssymb,amsthm,mathtools,enumitem}
\usepackage[margin=1in]{geometry}
\usepackage{tabularx}
\usepackage[hidelinks]{hyperref}
\newtheorem{theorem}{Theorem}[section]
\newtheorem{proposition}[theorem]{Proposition}
\newtheorem{lemma}[theorem]{Lemma}
\newtheorem{corollary}[theorem]{Corollary}
\theoremstyle{remark}

\newcommand{\Prob}{\mathbb P}
\newcommand{\E}{\mathbb E}
\newcommand{\one}{\mathbf 1}
\newcommand{\Id}{\mathrm{Id}}
\newcommand{\cS}{\mathcal S}
\newcommand{\cP}{\mathcal P}
\newcommand{\cI}{\mathcal I}
\newcommand{\TV}{d_{\mathrm{TV}}}
\newcommand{\astar}{a_*}

\DeclareMathOperator{\Var}{Var}

\DeclareMathOperator{\dist}{dist}
\hypersetup{pdftitle={Entropy and singularity of spectral and pivotal sets in planar percolation},pdfauthor={Yitzchak Shmalo}}
\title[Entropy and singularity in percolation]{Entropy and singularity of spectral and pivotal sets in planar percolation}
\author{Yitzchak Shmalo}
\address{Einstein Institute of Mathematics, The Hebrew University of Jerusalem, Jerusalem, Israel}
\date{}
\setlist{itemsep=3pt,topsep=5pt}
\begin{document}
\begin{abstract}
The spectral sample and pivotal set of a percolation crossing have the same one- and two-coordinate inclusion probabilities under the uniform product measure. We prove that both Shannon entropies are comparable to total influence for critical square-lattice bond and triangular-lattice site crossings, with matching estimates at every retained nonroot spatial resolution. The triangular entropy bounds are uniform over a class of inhomogeneous near-critical product measures. For critical triangular-site square crossings, the two laws are asymptotically mutually singular: an elementary triangular face is forbidden in the pivotal set and occurs with positive density in a large spectral sample. Every fixed face density below an explicit threshold has nonempty spectral probability comparable to $R^2\alpha_4(R)^2$, while the nonempty overlap of the two laws has decay exponent $11/12$. Their minimum expected symmetric-difference distance over all couplings is comparable to total influence. Geometric separation persists after independent thinning whenever the retention probability $\rho_R$ satisfies $\rho_R^3R^2\alpha_4(R)\to\infty$; the face detector fails when this quantity tends to zero. The proofs combine spatial encoding, arm estimates, local Fourier cancellation, and an exact six-cycle coercivity inequality preserved under exterior conditioning.
\end{abstract}
\maketitle
\enlargethispage{4pt}
\section{Introduction}

The spectral sample and the pivotal set are two random subsets associated with a Boolean crossing function. At the uniform product measure they have the same one- and two-coordinate inclusion probabilities \cite[equation~(2.13)]{GPS}. We prove that their Shannon entropies are also comparable to the same influence scale, with matching comparisons at each retained nonroot spatial resolution. For critical triangular-lattice square crossings, we further prove that their laws are asymptotically mutually singular. The distinction is detected by elementary triangular faces, which occur in a large spectral sample and are forbidden in a pivotal set.

The entropy estimate answers the square-crossing question in \cite[Section~11, Question~1]{GPS}. Its pivotal counterpart gives the other square-crossing bound in \cite[Conjecture~8]{Kalai}. The singularity theorem answers the discrete triangular-site case of \cite[Section~11, Question~5]{GPS}. Both questions appeared in the 2008 preprint of \cite{GPS}, published in 2010. The spectral entropy bound is a percolation instance of the Fourier entropy--influence conjecture of Friedgut and Kalai \cite{FK}; no assertion for arbitrary Boolean functions is made here. Likewise, the singularity result concerns discrete laws. It does not establish mutual singularity of their continuum limits or the square-lattice bond analogue.

\subsection{The random sets and their entropy}

Let $f:\{0,1\}^V\to\{-1,1\}$ be an increasing Boolean function under independent coordinate parameters $p_i\in(0,1)$. Use the normalized product characters
\[
 \chi_A(X)=\prod_{i\in A}\frac{X_i-p_i}{\sqrt{p_i(1-p_i)}},
 \qquad \widehat f(A)=\E[f\chi_A].
\]
The spectral sample $\cS_f$ has law
\begin{equation}\label{eq:spectral}
 \Prob(\cS_f=A)=\widehat f(A)^2.
\end{equation}
It has total mass one by Parseval. The pivotal set $\cP_f$ consists of the coordinates whose reversal changes $f$. Write
\[
 \iota_i(f)=\Prob(i\in\cP_f),\qquad
 I(f)=\E|\cP_f|,\qquad
 J(f)=\E|\cS_f|=\sum_i4p_i(1-p_i)\iota_i(f).
\]
The last identity is proved in \eqref{eq:biasedinfluence}. At the uniform measure, $I(f)=J(f)$ and
\begin{equation}\label{eq:jointmarginals}
 \Prob(i\in\cS_f)=\Prob(i\in\cP_f),\qquad
 \Prob(i,j\in\cS_f)=\Prob(i,j\in\cP_f)\quad(i\ne j).
\end{equation}
The two-coordinate identity is recalled in Lemma~\ref{lem:biasedpair}. For fair coordinates we also write $\omega_i=2X_i-1$, so $\chi_A=\prod_{i\in A}\omega_i$.

Shannon entropy $H$ and binary entropy $h_2$ are measured in bits. We write $\log_2$ for base-two logarithms and $\ln$ for natural logarithms. Let $Q$ be a rectangle with horizontal and vertical sides, larger side $R$, and aspect ratio at most a fixed $\kappa$. Its crossing function is
\[
 f=2\one_{\{Q\text{ has an open left--right crossing}\}}-1.
\]
We use the geometric tile convention of \cite[Section~2.1]{GPS}; coordinates are the random tiles meeting $Q$. Fix an admissible microscopic inner radius $r_0$ and put
\[
 m(t)=t^2\alpha_4(t),\qquad \alpha_4(t)=\alpha_4(r_0,t),
\]
where $\alpha_4$ is the critical alternating four-arm probability of the model. On the triangular lattice, $\alpha_4(R)=R^{-5/4+o(1)}$ \cite{SW}.

\begin{theorem}[Crossing entropy]\label{thm:main}
Fix $0<p_0\le1/2$, $\varepsilon\in(0,1/2)$, and $\kappa\ge1$.
Consider either of the following models, in the embeddings of
Section~\ref{sec:inputs}:
\begin{enumerate}[label=(\roman*)]
\item critical bond percolation on the square lattice;
\item triangular-lattice site percolation with independent parameters
\[
 p_x\in[p_*,1-p_*],\qquad p_*\in[p_0,1/2],\qquad
 R\le L_\varepsilon(p_*).
\]
Here $L_\varepsilon$ is the characteristic length
of~\cite[Section~3.1]{Nolin}, with $L_\varepsilon(1/2)=\infty$.
\end{enumerate}
There are constants $c,C,R_0>0$, depending only on these fixed
parameters and the lattice conventions, such that for every such
rectangle with $R\ge R_0$,
\begin{equation}\label{eq:mainentropy}
 cI(f)\le H(\cP_f)\le CI(f),\qquad
 cJ(f)\le H(\cS_f)\le CJ(f),\qquad
 I(f)\asymp J(f)\asymp m(R).
\end{equation}
The constants are uniform in the rectangle's location and the
admissible parameter profile.
\end{theorem}

To measure spatial entropy, project the tile centers coordinatewise
onto $Q$ and repeatedly bisect both sides of the rectangle. For
$s=R2^{-k}$, let $Z_s(A)$ record the depth-$k$ cells occupied by a set
$A$ of coordinates. The boundary convention and a fixed lattice
cutoff $s_*$ are specified in Section~\ref{sec:upper}. Scales
$s\ge s_*$ are called retained scales.

\begin{theorem}[Entropy at each resolution]\label{thm:scales}
Under the hypotheses of Theorem~\ref{thm:main}, for every retained
nonroot dyadic scale $s_*\le s=R2^{-k}\le R/2$ and
$A\in\{\cS_f,\cP_f\}$,
\begin{equation}\label{eq:maincoarse}
 H(Z_s(A))\asymp\E|Z_s(A)|\asymp\frac{m(R)}{m(s)}.
\end{equation}
At the root, $H(Z_R(\cP_f))\asymp1$, whereas
\begin{equation}\label{eq:rootexact}
 H(Z_R(\cS_f))=h_2\bigl((\E f)^2\bigr).
\end{equation}
For any deterministic offset of an axis-parallel square grid, the
entropy upper bound also holds at every $s_*\le s\le R$, using the
same projected centers.
\end{theorem}

Thus, away from the root, entropy is comparable to the expected number
of occupied cells. For standard critical square crossings,
$\E f_R\to0$~\cite[Section~1.3]{GPS}, so the spectral root entropy tends
to zero. Formula~\eqref{eq:rootexact} records this dependence on the
crossing probability.

\subsection{Local faces and singularity}

For the singularity statements, let $R$ be a positive integer and let $f_R$ be the critical triangular-site crossing sign of $[0,R]^2$. Write $\mu_R$ and $\nu_R$ for the laws of $\cS_R$ and $\cP_R$, and set
\[
 \epsilon_R=(\E f_R)^2=\mu_R(\varnothing).
\]
Let $V_R$ be the sites whose tiles meet the square, and let $K_R(A)$ count the elementary triangular faces contained in $A$ whose three tiles lie wholly inside the square. Let $V_R^\circ$ be the sites whose full seven-site wheel lies a fixed positive lattice distance inside the square, and put $X_R=|\cS_R\cap V_R^\circ|$. Every tested center is nonterminal and has all six ring neighbors. The fixed boundary collar may be enlarged without affecting the asymptotic conclusions.

\begin{theorem}[Discrete singularity]\label{thm:discrete}
For critical triangular-site square crossings,
\[
 \nu_R(K_R>0)=0,\qquad \mu_R(K_R>0)\longrightarrow1,
 \qquad \TV(\mu_R,\nu_R)\longrightarrow1,
\]
where total variation is normalized to lie in $[0,1]$.
\end{theorem}

In this model, Theorems~\ref{thm:main}, \ref{thm:scales}, and \ref{thm:discrete} therefore give
\[
 H(\cS_R)\asymp H(\cP_R)\asymp m(R),\qquad
 H(Z_s(\cS_R))\asymp H(Z_s(\cP_R))\asymp\frac{m(R)}{m(s)},
\]
at the retained nonroot scales, together with asymptotic singularity of the full laws. Equal first and second inclusion statistics and comparable entropy do not determine the typical local geometry.

Set
\begin{equation}\label{eq:astar}
 \astar=\frac{8-\sqrt{53}}{16}>\frac1{32},
\end{equation}
and define the binary relative entropy in natural units, for $0<a<1$, by
\[
 D(\theta\Vert a)=\theta\ln\frac{\theta}{a}
 +(1-\theta)\ln\frac{1-\theta}{1-a},
\]
with the usual zero-term conventions at $\theta=0,1$.

\begin{theorem}[Finite face bounds]\label{thm:finite}
For every $R$, $m>0$, and $0\le\theta<\astar$,
\begin{equation}\label{eq:density}
 \Prob\left(K_R(\cS_R)\le\frac{\theta}{3}X_R,\ X_R\ge m\right)
 \le\exp\left\{-\frac{m}{7}D(\theta\Vert\astar)\right\}.
\end{equation}
In particular,
\begin{equation}\label{eq:noface}
 \Prob(K_R(\cS_R)=0,\ X_R\ge m)\le(1-\astar)^{m/7}.
\end{equation}
These bounds hold for every fair site-connection function on a finite triangular-lattice domain, with the same complete nonterminal-wheel convention.
\end{theorem}

\begin{theorem}[Triangle-free mass]\label{thm:rate}
For critical triangular-site square crossings,
\begin{equation}\label{eq:rate}
 \mu_R(K_R=0,\ S\ne\varnothing)
 \asymp R^2\alpha_4(R)^2=R^{-1/2+o(1)}.
\end{equation}
Consequently $\mu_R(K_R=0)=\epsilon_R+R^{-1/2+o(1)}$.
\end{theorem}

The same sharp scale holds for every fixed $0\le\delta<\astar/3$: Theorem~\ref{thm:lowdensity} proves
\[
 \mu_R\bigl(S\ne\varnothing,\ K_R(S)\le\delta|S|\bigr)
 \asymp_\delta R^2\alpha_4(R)^2.
\]
Conditional on such low face density and nonemptiness, the spectral size has an exponentially decaying tail up to a vanishing boundary error.

\begin{theorem}[Total variation deficit]\label{thm:TVrate}
For the same square crossings,
\[
 1-\TV(\mu_R,\nu_R)=\epsilon_R+R^{-11/12+o(1)}.
\]
The second term is the overlap on nonempty supports.
\end{theorem}

The conformal crossing limit gives $\epsilon_R\to0$ \cite{Smirnov}. No rate for the empty atom is asserted. Theorem~\ref{thm:quads} extends discrete singularity to every fixed Jordan quad after conditioning the spectral sample to be nonempty; the corresponding unconditioned conclusion retains the common empty atom. Proposition~\ref{prop:biased} extends the finite face bounds to product measures with uniformly bounded ring parameters. Those finite bounds alone do not supply a near-critical asymptotic singularity theorem.

\subsection{The two uses of local irrelevance}

The entropy upper bound encodes occupied cells through successive spatial refinements. An occupied cell has one of fifteen nonempty four-child masks, so its contribution is bounded by a constant times its occupancy probability. Four-arm estimates in the interior and half-plane arm estimates near the boundary control the sum over cells. Positive relative growth of $m$ then sums the cost over scales. Pete proposed self-similarity and tree descriptions as a route to the percolation entropy bounds in 2011 \cite[slides~21--22]{Pete}; the spectral local-mass estimate in \cite[Lemma~3.2 and equation~(3.7)]{GPS} is a related antecedent.

For the entropy lower bound, a local configuration that makes a coordinate irrelevant contributes zero to any Fourier coefficient containing that coordinate. Bounded overlap permits simultaneous tests across a prescribed support and yields an exponential bound on each spectral atom. Closed circuits give a block version of the argument, and hence the lower bound at each retained nonroot spatial resolution. Monotonicity supplies a direct lower bound for the entropy of the pivotal set.

The singularity proof uses a stronger consequence of the same local independence of the crossing function. At a triangular-lattice site, the derivative vanishes whenever the open neighbors form one cyclic interval. Evaluation on these ring configurations is injective on the span of independent supports of the six-cycle. The resulting coercivity inequality bounds the Fourier mass of supports containing an incident triangle. Its zero identity survives every linear operator outside the complete wheel, including exact exterior Fourier conditioning. Disjoint wheels, sequential conditional bounds, and a seven-color partition amplify the local estimate. The pivotal prohibition follows from a path shortcut across a triangle. For geometric Fourier projections and conditional amplification in percolation, see \cite[Lemma~4.3 and Propositions~5.1 and~6.1]{GPS}.

The local entropy estimate controls the masses of individual atoms; the wheel estimate controls mass on a geometric class of supports. These are different conclusions from a common finite-product method. The critical arm estimates then place both conclusions on the percolation scale. No implication from entropy comparison to singularity is used.

Fourier methods for noise sensitivity and their percolation applications originate in \cite{BKS}; quantitative developments include \cite{SS,GPS,TV}, with background in \cite{GS}. General entropy estimates and structured Boolean classes are treated in \cite{KMS,KKLMS,Han,OWZ,OT,CKLS,WWW,GMP}. The spatial argument here removes the logarithmic loss that coordinate subadditivity can retain.

\subsection{Stability of the separation}

The local statistic gives more than a separating event. Theorem~\ref{thm:distance} shows that, under every coupling of the two laws, a fixed positive fraction of the spectral sites must be changed to obtain a pivotal realization, with probability tending to one. The minimum expected number of changed sites is $\asymp m(R)$. If each spectral site is independently retained with probability $\rho_R$, Proposition~\ref{prop:thinning} proves that the face event still separates the thinned law from every face-free law when $\rho_R^3m(R)\to\infty$. When $\rho_R^3m(R)\to0$, that event has probability tending to zero. Thus the face detector has retention threshold scale $m(R)^{-1/3}=R^{-1/4+o(1)}$ in these two regimes. Corollary~\ref{cor:testing} records the nonempty testing error and the corresponding lower bound on relative entropy.

Sections~\ref{sec:forest}--\ref{sec:mesoscopic} prove the entropy and spatial-resolution estimates. Sections~\ref{sec:localgeometry} and~\ref{sec:faces} establish the local separator and its finite-volume bounds. Section~\ref{sec:criticalrates} proves the critical square rates, including the sharp cost of low face density. Section~\ref{sec:localscope} gives the abstract criterion and the product-measure and quad extensions. Section~\ref{sec:perturbations} proves the perturbation and information consequences. Section~\ref{sec:scope} records the scope of the results and the remaining questions.

\section{Entropy under spatial refinement}\label{sec:forest}

Let a finite rooted forest partition a finite coordinate set $V$ into
nonempty cells: the root cells partition $V$, and the children of each
internal vertex partition its cell. For a random subset $A\subseteq V$,
write $O_v=\one_{\{A\cap C_v\ne\varnothing\}}$ and $p_v=\Prob(O_v=1)$.
An occupied vertex with $d_v$ children has at most $2^{d_v}-1$ nonempty
child masks. The chain rule therefore bounds entropy by expected
occupancy counts, without any independence assumption.

\begin{lemma}[Refinement cost]\label{cor:forestcost}
Let $O_{\rm rt}$ be the root occupancy vector. For arbitrary leaf depths,
\[
 H(A)\le H(O_{\rm rt})
 +\sum_{v\ \mathrm{internal}}p_v\log_2(2^{d_v}-1)
 +\sum_{v\ \mathrm{leaf}}p_v\log_2(2^{|C_v|}-1).
\]
In particular, suppose terminal cells have at most $M\ge1$ elements and
every internal cell has at most $b\ge1$ children. For a single tree with
all terminal cells at depth $L$, let $Z_j$ be its occupied-cell pattern
at depth $j$ and $N_j=|Z_j|$. Then, for $0\le k\le L$,
\begin{align}
 H(A)&\le h_2(\Prob(A\ne\varnothing))
 +\log_2(2^b-1)\sum_{j=0}^{L-1}\E N_j
 +\log_2(2^M-1)\E N_L,\label{eq:treefull}\\
 H(Z_k)&\le h_2(\Prob(A\ne\varnothing))
 +\log_2(2^b-1)\sum_{j=0}^{k-1}\E N_j.\label{eq:treecoarse}
\end{align}
\end{lemma}
\begin{proof}
List the root occupancy vector and then the child masks, with each
parent preceding its descendants. The occupancy of a vertex is known
when its mask is exposed. Its mask is empty when it is unoccupied and
has at most $2^{d_v}-1$ possible values when it is occupied. The chain
rule bounds this contribution by $p_v\log_2(2^{d_v}-1)$. At each leaf,
record the contents of its cell; the analogous contribution is at most
$p_v\log_2(2^{|C_v|}-1)$. These variables determine $A$, proving the
general bound. Grouping vertices by depth gives \eqref{eq:treefull},
and stopping at depth $k$ gives \eqref{eq:treecoarse}.

Vertices with $p_v=0$ contribute zero, as do unary refinement steps.
\end{proof}

Let $m$ be positive on the available spatial scales. A lower power bound
with fixed $c,\epsilon>0$,
\begin{equation}\label{eq:massgrowth}
 m(t)\ge c(t/s)^\epsilon m(s)\qquad(t\ge s),
\end{equation}
implies
\begin{equation}\label{eq:dyadicsum}
 \sum_{j\ge0}\frac1{m(2^js)}\le\frac{C}{m(s)}
\end{equation}
with the sum truncated when only finitely many scales are available.
Indeed, each summand is at most $c^{-1}2^{-j\epsilon}/m(s)$, so one may
take $C=c^{-1}(1-2^{-\epsilon})^{-1}$.
The bound \eqref{eq:dyadicsum} is the scale summability used below.

\section{Lower bounds from local configurations}\label{sec:lower}

We first prove two finite-product entropy bounds, using pivotal fibers and local configurations that make a coordinate irrelevant.

\begin{theorem}[Pivotal entropy under product measures]\label{thm:pivlower}
Let $f:\{0,1\}^V\to\{-1,1\}$ be increasing, and let the coordinates be independent with parameters $p_i\in(0,1)$. Write $\cP_f$ for its pivotal set and $\iota_i=\Prob(i\in\cP_f)$. Then
\begin{equation}\label{eq:pivlower}
 H(\cP_f)+H(f\mid\cP_f)\ge\sum_{i\in V}h_2(p_i)\iota_i,
\end{equation}
where $h_2$ is binary entropy. In particular,
\[
 H(\cP_f)\ge\sum_i h_2(p_i)\iota_i-H(f).
\]
For unbiased coordinates this gives $H(\cP_f)\ge I(f)-1$.
\end{theorem}
\begin{proof}
Monotonicity implies $X_i=(f+1)/2$ whenever $i$ is pivotal. Consequently, for every $S\subseteq V$,
\[
 \Prob(\cP_f=S,f=1)\le\prod_{i\in S}p_i,
 \qquad
 \Prob(\cP_f=S,f=-1)\le\prod_{i\in S}(1-p_i).
\]
Taking logarithmic reciprocals of the positive joint atoms and averaging gives
\begin{align*}
 H(\cP_f,f)
 &\ge\sum_i\Bigl[\Prob(i\in\cP_f,f=1)\log_2(1/p_i)\\
 &\hspace{38mm}+\Prob(i\in\cP_f,f=-1)\log_2(1/(1-p_i))\Bigr].
\end{align*}
The event that $i$ is pivotal depends only on the other coordinates. On this event $f=2X_i-1$. Independence therefore gives probabilities $p_i\iota_i$ and $(1-p_i)\iota_i$ in the last display. Its right-hand side is $\sum_i h_2(p_i)\iota_i$. The chain rule proves \eqref{eq:pivlower}; conditioning reduces entropy and $H(f)\le1$.
\end{proof}

For the spectral bound, a \emph{guard} for coordinate $i$ is an event determined by other coordinates on which changing coordinate $i$ cannot change the function. The following statement also applies to nonmonotone functions.

\begin{theorem}[Local guards and Fourier coefficients]\label{thm:guards}
Let $f:\{0,1\}^U\to\{-1,1\}$ depend only on $V\subseteq U$, under a product measure with $p_i\in[p_0,1-p_0]$, where $0<p_0\le1/2$. For each $i\in V$, suppose that a specified pattern $E_i$ on a set $G_i\subseteq U\setminus\{i\}$ makes $f$ independent of coordinate $i$. For positive integers $D,K$, assume $|G_i|\le D$, and that every set $F_i=G_i\cup\{i\}$ intersects at most $K$ of the sets $F_j$, counting itself. Alternatively, assume that the intersection graph of the sets $F_i$ admits a proper coloring with at most $K$ colors. With normalized product Fourier characters,
\begin{equation}\label{eq:guardcoeff}
 \widehat f(S)^2\le
 \left(\prod_{j\in S}4p_j(1-p_j)\right)(1-p_0^D)^{2|S|/K}
 \quad(S\subseteq V).
\end{equation}
It follows that
\begin{equation}\label{eq:guardentropy}
 H(\cS_f)\ge\sum_{j\in V}
 \left(c_g+\log_2\frac1{4p_j(1-p_j)}\right)
 \Prob(j\in\cS_f),\qquad
 c_g=\frac{2}{K}\log_2\frac1{1-p_0^D}>0.
\end{equation}
The parameters $D$ and $K$ may be replaced by any positive upper bounds.
\end{theorem}
\begin{proof}
Fix $S$. Under the intersection-count hypothesis, choose $T\subseteq S$ greedily with pairwise disjoint $F_i$; at most $K$ candidates are removed at each choice. Under the coloring hypothesis, choose the largest color class in $S$ instead. In either case, $|T|\ge|S|/K$ and the chosen sets are pairwise disjoint. Put
\[
 \chi_S=\prod_{j\in S}\frac{X_j-p_j}{\sqrt{p_j(1-p_j)}},
 \qquad M=\prod_{i\in T}(1-\one_{E_i}).
\]
Expanding the product defining $M$ gives
\begin{equation}\label{eq:guardcancel}
 \widehat f(S)=\E[f\chi_S M].
\end{equation}
Indeed, a nonempty expansion term contains $\one_{E_i}$ for some $i\in T$. On $E_i$, $f$ is independent of $X_i$; every guard indicator in that term is also independent of $X_i$, by disjointness of the $F_j$. Averaging the mean-zero factor $\chi_i$ annihilates the term.

Put $a_S=\E|\chi_S|=\prod_{j\in S}2\sqrt{p_j(1-p_j)}$, and define
\[
 d\nu_S=\frac{|\chi_S|}{a_S}\,d\Prob.
\]
This is a product probability measure with parameter $1/2$ on $S$ and the original parameters elsewhere: for $j\in S$, the two masses $p_j|\chi_j(1)|$ and $(1-p_j)|\chi_j(0)|$ both equal $\sqrt{p_j(1-p_j)}$. Thus every specified guard pattern has $\nu_S$-probability at least $p_0^D$, and the guards indexed by $T$ are independent under $\nu_S$. Equation~\eqref{eq:guardcancel} gives
\[
 |\widehat f(S)|\le\E[|\chi_S|M]
 =a_S\nu_S(M=1)\le a_S(1-p_0^D)^{|T|}.
\]
Square this estimate and use $|T|\ge|S|/K$ to obtain \eqref{eq:guardcoeff}. Taking logarithmic reciprocals of the positive spectral atoms and summing proves \eqref{eq:guardentropy}.
\end{proof}

\begin{corollary}[Connectivity on bounded-degree graphs]\label{cor:guardcross}
Consider site or bond connectivity between two fixed vertex sets under independent parameters in $[p_0,1-p_0]$. Suppose the underlying graph has bounded degree, and that a single isolated site, respectively a single isolated edge, cannot join the two distinguished sets. Then
\[
 H(\cS_f)\ge c\sum_i4p_i(1-p_i)\iota_i,
\]
with $c>0$ depending only on $p_0$ and the degree bound.
\end{corollary}
\begin{proof}
For a site $i$, close all its neighboring sites. If $i$ is open, its occupied component is then a singleton. For an edge $i=uv$, close all other edges incident to either $u$ or $v$; its open component then consists of at most that edge and its endpoints. In either case the stated separation assumption makes the guarded coordinate irrelevant. The sizes and overlap multiplicities of these local neighborhoods are bounded in terms of the degree. Add finitely many irrelevant neighboring coordinates to $U$ if necessary. Theorem~\ref{thm:guards} applies. Finally, conditional variance in coordinate $i$ is $4p_i(1-p_i)$ on the pivotal event and zero otherwise. Parseval gives
\begin{equation}\label{eq:biasedinfluence}
 \Prob(i\in\cS_f)=4p_i(1-p_i)\iota_i.
\end{equation}
\end{proof}

For triangular-lattice site percolation and square-lattice bond percolation, each guard uses at most six coordinates. In the bond model the coordinate adjacency graph is the line graph of the square lattice, of maximum degree six. A graph of maximum degree six has at most $1+6+6\cdot5=37$ vertices within distance two, so $D=6$ and $K=37$ are valid. In the geometric bond convention of \cite{GPS}, the isolated white component consists of the variable edge tile and the two deterministic white endpoint squares. Its $\ell^\infty$ diameter is $3/2$ in the unit embedding, so horizontal side length at least $2$ is a sufficient separation cutoff, including at boundary tiles. At $p=1/2$, for a square larger than the fixed tile cutoff,
\[
 H(\cS_R)\ge\frac{2}{37}\log_2\frac{64}{63}\,I(f_R).
\]
For triangular sites the constant improves to
\begin{equation}\label{eq:triangular_guard_constant}
 H(\cS_R)\ge\frac{2}{7}\log_2\frac{64}{63}\,I(f_R).
\end{equation}
Indeed, color axial lattice coordinates $(i,j)$ by $i+3j\pmod7$. The seven vertices of a closed neighborhood have different colors, since the neighbor increments are $\pm1,\pm2,\pm3$ modulo seven. Two centers whose neighborhoods intersect lie in a common closed neighborhood, so they have different colors. The coloring alternative in Theorem~\ref{thm:guards} therefore applies with $K=7$. The cancellation in \eqref{eq:guardcancel} is related to the restricted Fourier-coefficient argument in \cite[Lemma~4.3]{GPS}. Here bounded overlap permits guards to be selected from an arbitrary Fourier support. Chang's lower bound for the common-bias product measure $\mu_p^{\otimes n}$~\cite{Chang} uses a different influence transform and degenerates at the unbiased parameter; the local guard hypothesis supplies a positive bound there.

\section{Arm estimates and occupied cells}\label{sec:upper}

Let $Q$ be a translate of $[0,a]\times[0,b]$ in Euclidean coordinates, with $R=\max(a,b)$ and $\min(a,b)\ge R/\kappa$. The constants in this section may depend on the fixed aspect bound $\kappa$. Let $f$ be its left--right crossing function. For a coordinate block $B$, write $\Lambda_B$ for the event that the outside configuration leaves $f$ nonconstant as the coordinates in $B$ vary.

\begin{lemma}[Block variance]\label{lem:blockvariance}
Under any product measure and for any Boolean function,
\begin{equation}\label{eq:blockvariance}
 \Prob(\cS_f\cap B\ne\varnothing)
 =\E\Var(f\mid X_{B^c})\le\Prob(\Lambda_B),
 \qquad
 \Prob(\cP_f\cap B\ne\varnothing)\le\Prob(\Lambda_B).
\end{equation}
\end{lemma}
\begin{proof}
Conditional expectation onto the outside coordinates removes exactly the orthonormal product characters whose supports meet $B$. Parseval identifies the lost squared norm with the spectral probability in \eqref{eq:blockvariance}. Conditional variance is zero off $\Lambda_B$ and at most one everywhere. A pivotal coordinate in $B$ supplies two choices of its contents with different function values.
\end{proof}

Write $\alpha_4(r,t)$ for the critical alternating four-arm probability of the model, and set $\alpha_4(t)=\alpha_4(r_0,t)$ for a fixed microscopic radius $r_0$, with $\alpha_4(t)=1$ for $t<r_0$. Put
\[
 m(t)=t^2\alpha_4(t).
\]
The critical estimates used below are fixed-factor regularity, quasi-multiplicativity, and
\begin{equation}\label{eq:armgap}
 \alpha_4(u,t)\ge c(u/t)^{2-\eta},\qquad t\ge u\ge r_0,
\end{equation}
for some $\eta>0$. They imply \eqref{eq:massgrowth} for $m$. These facts hold in both critical models considered here; see \cite[Section~2.2]{GPS}.

Assume that the given product measure satisfies the following estimates, uniformly in location at the scales under consideration:
\begin{align}
 \Prob(A_4(r,t))&\le C\alpha_4(r,t),\label{eq:armactual}\\
 \Prob(A_3^H(r,t))&\le C(r/t)^\beta,\qquad\beta>1.\label{eq:halfarm}
\end{align}
Here $A_3^H$ is the alternating three-arm event in any half-plane parallel to a side of $Q$, with both color orders allowed. Fixed microscopic changes of the boundary and fixed-factor changes of radii are permitted.

\subsection{The boundary contribution}

Choose $D_0$ larger than the maximum distance from a tile point to its center. Let $V$ consist exactly of the random tiles meeting $Q$. Project each center coordinatewise onto $Q$, and bisect both sides of $Q$ repeatedly. Use lower-left-inclusive cells, with the outer top and right sides included. At depth $k$, write $s=R2^{-k}$; cell side lengths lie in $[s/\kappa,s]$. Let $J_C\subseteq V$ be the coordinates assigned to cell $C$, and put
\[
 Z_s(A)=\{C:A\cap J_C\ne\varnothing\},\qquad N_s(A)=|Z_s(A)|.
\]
The sets $J_C$ partition all relevant coordinates, including those whose centers lie outside $Q$. Terminate the partition at scale $s_L\in[s_*,2s_*)$, for a fixed sufficiently large $s_*$; every terminal cell then has at most a fixed number $M$ of coordinates.

For the center $z$ of a cell $C$, define $B_C=V\cap\{x:\|x-z\|_\infty\le2s\}$, identifying a coordinate with its tile center. At retained scales $s\ge s_*$, $J_C\subseteq B_C$. Assign $z$ to a nearest side $e$ of $Q$, with deterministic ties, and let $y$ be its perpendicular projection onto $e$. Set
\[
 h=\operatorname{dist}(z,e),\qquad
 t=\min\{|y-v_1|,|y-v_2|,w_e\},
\]
where $v_1,v_2$ are the endpoints of $e$ and $w_e$ is its distance to the opposite side.
Then $h\le t\le R$. Fix a large constant $C_0$ and define
\[
 d=\min\{R,\max(C_0s,h)\},\qquad
 D=\min\{R,\max(d,t)\}.
\]

\begin{lemma}[Side-distance bound]\label{lem:sidebound}
Under \eqref{eq:armactual}--\eqref{eq:halfarm},
\begin{equation}\label{eq:sidebound}
 \Prob(\Lambda_{B_C})\le C\alpha_4(s,d)(d/D)^\beta.
\end{equation}
\end{lemma}
\begin{proof}
First suppose $s\le R/C_0$. Fill the block closed, and let $W_0$ be the union of the closures of the open tiles, including the deterministic white vertex tiles in the bond model. On $\Lambda_{B_C}$ this set has no left--right crossing of $Q$. Planar separation for the finite polygonal set $W_0\cap Q$ gives a simple top--bottom curve in $Q\setminus W_0$. Use this curve as the closed top--bottom crossing. Fill the block open and choose a simple open left--right crossing in the resulting union $W_1$. These curves must intersect in the rectangle, and can intersect only in changed tiles, since outside those tiles $W_0=W_1$. Every changed tile lies inside a square about $z$ of radius $K_0s$, for a fixed $K_0>2$.

Convert each complete crossing to its tile walk and loop-erase that walk before extracting its tails. Each open tile clipped by $Q$ is convex. In an initially closed tile, $W_0$ removes only whole boundary faces or vertices, so the remaining portion clipped by $Q$ is also convex. The retained walks therefore have realizations in their original colored sets, with the closed realization contained in $Q\setminus W_0$. On the triangular lattice they give simple site paths. In the bond model, use incidence walks through random edge tiles and deterministic vertex squares for the open crossing, and through random edge tiles and deterministic face squares for the closed crossing; their loop erasures give simple primal and dual paths. These coordinate paths remain within a fixed tile distance of their geometric realizations. The retained realizations still join opposite sides, so they still intersect in the changed-tile square.

Enlarge the changed-tile square by that fixed distance. Follow the resulting paths from each of their distinguished sides to their first encounter with the square. Same-color exterior tails use disjoint coordinates by simplicity, and opposite colors cannot share a coordinate outside the changed block. In the interior, their outer endpoints have the order left, top, right, bottom. Disjoint planar arcs preserve this order on the inner boundary, so the four colors alternate.

When $h$ is a sufficiently large multiple of $s$, these arms cross an annulus about $z$ from radius $Gs$ to $h/G$, for fixed large $G$. For the boundary range, a square of radius $Gd$ about $y$ contains the changed tiles and, when $t/d$ is sufficiently large, meets only the side $e$. Follow the crossings from the other three sides until their first encounter with this boundary notch, omitting the tail to $e$. The three disjoint tails have the order of those three sides on the notch, hence alternate. Each reaches radius $t/G$, since the other sides are at distance at least $t$ from $y$. They give the required half-annulus event. Shift its boundary outward by a fixed tile distance to include the coordinate paths. Increasing $G$ and $s_*$ absorbs the shift and all tile-to-center errors.

The full annular support lies within distance $h+h/G+O(1)$ of $y$, while the half-annular support begins at distance $Gd-O(1)$, with $d\ge h$. The supports are therefore disjoint. Independence, \eqref{eq:armactual}, and \eqref{eq:halfarm} bound their intersection by $C\alpha_4(s,d)(d/D)^\beta$. If either annulus has bounded ratio, omit it and absorb its factor by fixed-factor regularity. For near-critical measures, first extract smaller annuli in the lattice norm of the arm estimates, with outer radii strictly less than $R$. Apply the estimates for the given measure at these radii, then use critical fixed-factor regularity to compare with the displayed radii. Finally, if $s>R/C_0$, then $d=D=R$ and $R/s<C_0$; the trivial probability bound suffices.
\end{proof}

\begin{proposition}[Occupancy estimate]\label{prop:occupancy}
For $A=\cS_f$ or $\cP_f$,
\begin{equation}\label{eq:occupancyupper}
 \E N_s(A)\le C(R/s)^2\alpha_4(s,R)
 \asymp\frac{m(R)}{m(s)}.
\end{equation}
The same bound holds for the sum of the block-oscillation probabilities $\sum_C\Prob(\Lambda_{J_C})$.
\end{proposition}
\begin{proof}
Group cells by their assigned side and by dyadic classes $d\asymp u$, $D\asymp v$, where $s\lesssim u\le v\lesssim R$. For a fixed side, their centers lie in a strip of perpendicular thickness $O(u)$. Its available tangential length is $O_\kappa(v)$: either an endpoint of the side is within $O(v)$, or the opposite-side distance is $O(v)$, in which case the entire side has length $O_\kappa(v)$. Since centers have spacing at least $s/\kappa$ in both coordinates, the class contains at most
\[
 C_\kappa(u/s+1)(v/s+1)\le C_\kappa uv/s^2
\]
cells. The count includes corner cells, and the opposite-side term $w_e$ in $t$ accounts for the case in which the whole side contributes.

Block enlargement, Lemmas~\ref{lem:blockvariance} and \ref{lem:sidebound}, and the count give
\begin{align*}
 \E N_s(A)
 &\le C\sum_u\sum_{v\ge u}\frac{uv}{s^2}\alpha_4(s,u)(u/v)^\beta\\
 &\le C\sum_u(u/s)^2\alpha_4(s,u),
\end{align*}
because $\sum_{v\ge u}(u/v)^{\beta-1}$ is a bounded geometric series. Quasi-multiplicativity and \eqref{eq:armgap} imply
\[
 (u/s)^2\alpha_4(s,u)
 \le C(u/R)^\eta (R/s)^2\alpha_4(s,R).
\]
Summing over $u$ proves \eqref{eq:occupancyupper}. The same estimates directly bound $\sum_C\Prob(\Lambda_{J_C})$. The final scale comparison follows from quasi-multiplicativity.
\end{proof}

\subsection{Entropy upper bounds}

\begin{theorem}[Arm criterion for entropy]\label{thm:armcriterion}
Assume the preceding arm estimates. Then
\begin{equation}\label{eq:armcriterion}
 H(\cS_f),\ H(\cP_f)\le C m(R),\qquad
 H(Z_s(A))\le C m(R)/m(s)
\end{equation}
at every retained dyadic scale, provided $R$ exceeds a fixed microscopic cutoff. If $I(f)\ge c m(R)$, then $H(\cS_f),H(\cP_f)\le CI(f)$. If this lower bound holds and all coordinate parameters lie in $[p_0,1-p_0]$, then $H(\cS_f)\le C\E|\cS_f|$.
\end{theorem}
\begin{proof}
Apply Lemma~\ref{cor:forestcost} with $b=4$ and terminal capacity $M$. Proposition~\ref{prop:occupancy} bounds each occupied-cell expectation by $Cm(R)/m(s)$. The reciprocal sum \eqref{eq:dyadicsum} is bounded above the fixed terminal scale; the terminal term has the same bound. This gives $H(A)\le1+Cm(R)$. For the pattern at scale $s$, stop the tree at that level and obtain $H(Z_s)\le1+Cm(R)/m(s)$. The growth estimate gives $m(R)/m(s)\ge c$, and also $m(R)\ge c$ above the cutoff, so both root terms are absorbed. The influence formulations follow from the extra lower bound and \eqref{eq:biasedinfluence}.
\end{proof}

\begin{corollary}[Other resolution grids]\label{cor:grids}
The entropy upper bound holds for any deterministic offset of an axis-parallel square grid of side $r$, $s_*\le r\le R$, using the same projected centers:
\[
 H(Z_r^{\mathrm{grid}}(A))\le C m(R)/m(r).
\]
\end{corollary}
\begin{proof}
For $r\ge2s_*$ choose a retained dyadic parameter $s$ with $r/2<s\le r$. Each rectangular partition cell has both side lengths at most $r$ and meets at most four target cells. Given its occupancy, its contribution to the target pattern can be recorded by one of at most fifteen nonempty masks. Therefore
\[
 H(Z_r^{\mathrm{grid}})\le H(Z_s)+\log_2(15)\E N_s
 \le C m(R)/m(r),
\]
using fixed-factor regularity. The bounded remaining range follows by data processing from $H(A)\le Cm(R)$.
\end{proof}

\section{Uniform inputs for the percolation models}\label{sec:inputs}

We use the following two regimes. In both, rectangles have horizontal and vertical sides in the displayed Euclidean embedding and aspect ratio at most $\kappa$.
\begin{enumerate}[label=(\roman*)]
\item Critical bond percolation on $\mathbb Z^2$, with parameter $1/2$.
\item Site percolation on
\[
 \mathbb T=\{(i+j/2,\sqrt3j/2):i,j\in\mathbb Z\},
\]
with independent parameters $p_x\in[p_*,1-p_*]$, where $p_*\in[p_0,1/2]$, $p_0>0$ is fixed, and $R\le L_\varepsilon(p_*)$. Here $L_\varepsilon$ is the characteristic length defined in \cite[Section~3.1]{Nolin}, with fixed $\varepsilon\in(0,1/2)$ and $L_\varepsilon(1/2)=\infty$.
\end{enumerate}
The second regime includes critical triangular percolation and spatially varying near-critical parameters. Constants are uniform over the indicated parameter profiles, rectangle locations, and retained scales; they may depend on $p_0,\varepsilon,\kappa$. All statements concern $R\ge R_0$ for a fixed sufficiently large lattice cutoff. Numbered citations to Nolin~\cite{Nolin} refer to the published version.

To pass from lattice-coordinate RSW to the Euclidean rectangles used here, cover each fixed-aspect-ratio corridor by a bounded number of overlapping lattice-coordinate parallelograms of side at most $cR$, where $c>0$ is fixed and sufficiently small. Crossing and joining these parallelograms gives the required corridor crossings by RSW and positive association. Thus each near-critical RSW input is applied at a lattice scale below $L_\varepsilon(p_*)$, even after the change of coordinates.

Critical RSW, arm separation, quasi-multiplicativity, and \eqref{eq:armgap} are available in regime (i) from \cite[Section~2.2]{GPS}. The half-plane three-arm bound \eqref{eq:halfarm} holds with $\beta=2$ by \cite[equation~(4.12)]{GPS}; lattice rotations give all four side directions. For regime (ii), Nolin's sandwich-product formulation, arm separation, and arm comparison \cite[Definition~5, Theorems~11 and~27, Proposition~17]{Nolin} give uniform RSW and \eqref{eq:armactual}. The next lemma gives a half-plane estimate uniform in the boundary direction in both regimes.

\begin{lemma}[Independent central inversion]\label{lem:reflection}
In either regime, let $L=\infty$ for critical square-lattice bonds and
$L=L_\varepsilon(p_*)$ for the triangular-lattice sandwich product measures.
For every Euclidean half-plane $H$, every $z\in\partial H$, either alternating
three-arm color order, and $r_*\le r<t\le L$,
\begin{equation}\label{eq:reflection}
 \Prob\bigl(A_3^H(z;r,t)\bigr)
 \le C(r/t)^{1+\eta_6/2},
\end{equation}
where $\eta_6>0$ is a critical alternating six-arm gap for the model:
$\alpha_{6,\mathrm{crit}}(a,b)\le C(a/b)^{2+\eta_6}$.
The constants are uniform in the direction and position of the boundary,
including its microscopic phase.  Any fixed microscopic padding of the
half-plane or of the annulus is allowed, with constants depending on that
padding.  Bond arms use open primal paths and closed dual paths.
\end{lemma}
\begin{proof}
Fix one of the two orders and denote its event by $E$.
Write $H=\{x:n\cdot(x-z)\ge0\}$, where $n$ is its inward unit normal.
Identify a random coordinate with its site, or with its primal-edge midpoint.
Choose a fixed $d_0$ so that every coordinate used by $E$ belongs to
\[
 U=\{x:n\cdot(x-z)\ge-d_0\}.
\]
Here and below a geometric set also denotes the random coordinates it contains.
Let $\ell$ exceed the distance from a coordinate to any point of its site-path
or primal/dual-edge realization, including all fixed endpoint and padding
conventions.  Choose a central inversion $g(x)=2c-x$ satisfying
\begin{equation*}
 |c-z|\le D_0,\qquad
 n\cdot(c-z)<-d_0-2\ell-2,
\end{equation*}
where $D_0$ is fixed independently of $n,z$.

For triangular sites, take $c$ in $\mathbb T$.  The lattice has bounded
covering radius, so such a $c$ exists by choosing it close to a point a fixed
distance from $z$ in direction $-n$.  Then $g$ is an involutive lattice
isometry, and the coordinate map is $\tau(x)=g(x)$.

For square bonds, take
\[
 c\in(1/4,1/4)+(1/2)\mathbb Z^2,
 \qquad 2c=a+(1/2,1/2),\quad a\in\mathbb Z^2.
\]
This set of centers also has bounded covering radius.  The map
$g(x)=a+(1/2,1/2)-x$ interchanges the primal lattice and its dual.  For a
primal edge $e$, define
\[
 \tau(e)=g(e)^*,
\]
where the star denotes the perpendicular primal edge crossing the dual edge
$g(e)$.  Taking a crossing edge commutes with $g$, so $\tau^2(e)=e$.
Thus $\tau$ is a bijection of the random bond coordinates, with midpoint
$g(\operatorname{mid}(e))$.  In particular, it exchanges horizontal and
vertical coordinates.  Moreover,
$g(e^*)=\tau(e)$: a primal path is carried to a dual path, and a dual path
to a primal path.

In both models the coordinate sets $U$ and $\tau(U)$ are disjoint.  Indeed,
their geometric half-planes are separated by more than $4\ell+4$.
The same separation, after allowing distance $\ell$ from the coordinates,
keeps all realizations of the two groups of paths disjoint.
Define a product measure $\nu$ by retaining $p_x$ on $U$, setting the
parameter at $\tau(x)$ to $1-p_x$ for every $x\in U$, and using $1/2$
on the remaining coordinates.  This remains in the triangular sandwich
class; in the bond model it is just the critical fair product measure.

Let $E'$ be the $\tau$-image of $E$ with every coordinate state reversed.
It is supported on $\tau(U)$ and has $\nu(E')=\Prob(E)$.
For bonds the state reversal gives exactly the required geometric colors:
an open primal path becomes a closed dual path, and a closed dual path
becomes an open primal path.  Independence now gives
\begin{equation}\label{eq:reflectsquare}
 \Prob(E)^2=\nu(E\cap E').
\end{equation}

On $E\cap E'$, truncate the six paths to a common annulus. Choose a lattice
center $q$ within bounded distance of $z$; both $z$ and $g(z)$ are then
within a fixed distance $D$ of $q$. Let $S_a(q)$ denote the balls in the
lattice norm used for the arm estimates. There is a fixed $A$ such that
\[
 B(q,a/A)\subseteq S_a(q)\subseteq B(q,Aa).
\]
Choose a sufficiently large fixed $K$, and increase $r_*$ to absorb $D$,
all lattice endpoint conventions, and the minimum admissible inner radius
$n_0(6)$ for the six-arm comparison.
When $t\ge4K^2r$, put
\[
 a=\lceil Kr\rceil+c_{\rm lat},\qquad
 b=\lfloor t/K\rfloor-c_{\rm lat},
\]
where the fixed integer $c_{\rm lat}$ absorbs the endpoint conventions of
the arm estimates.
Every original arm has a point inside $S_a(q)$ and reaches outside
$S_b(q)$. Retain the segment from its last visit to the inner boundary
before its first visit to the outer boundary, up to that outer visit.
Fixed lattice endpoint adjustments are absorbed by $c_{\rm lat}$.
This gives six disjoint arms in the lattice annulus, with
\[
 n_0(6)\le a<b\le t/K\le L,
 \qquad a/b\le C r/t;
\]
increasing $r_*$ once more ensures these inequalities despite rounding.

Their colors alternate: the two separated half-planes meet the outer boundary
of the common convex ball in two disjoint arcs.  Each group of three therefore occurs
consecutively in cyclic order.  Disjoint paths preserve the linear order
within a group under the annular truncation.  Its word is $BWB$, while
the other group's word is $WBW$, or conversely.  Reversing a group's
linear orientation does not change either word.  Their cyclic
concatenation is consequently alternating.

The critical six-arm gap holds in both models
\cite[Remark~4.6]{GPS}.  In the triangular regime,
\cite[Theorem~27]{Nolin} compares this alternating six-arm event under
$\nu$ to criticality at the radii $a,b$. The bound $b\le t/K\le L$
ensures that this comparison is within the characteristic-length cutoff.
Thus, in either model,
\[
 \Prob(E)^2\le C\alpha_{6,\mathrm{crit}}(a,b)
 \le C(a/b)^{2+\eta_6}
 \le C(r/t)^{2+\eta_6}.
\]
Taking square roots proves \eqref{eq:reflection}; the bounded-ratio
range follows from probability at most one.  If both original orders
are allowed, apply the argument separately and add the two bounds.
Increasing the fixed constants proves the padded versions as well.
\end{proof}

\begin{corollary}[Half-plane arms with unspecified order]\label{cor:generalreflection}
Fix an integer $j\ge1$.  In regime (ii), for every direction of $H$ and
$z\in\partial H$,
\[
 \Prob\bigl(A_j^{H,\mathrm{any}}(z;r,t)\bigr)
 \le C_j\sqrt{\alpha_{2j,\mathrm{crit}}(r,t)},
 \qquad r_*(j)\le r<t\le L_\varepsilon(p_*),
\]
where the $j$ site arms are disjoint and monochromatic but their colors
and order are unspecified.  The critical event on the right is alternating.
The constants have the same direction, phase, and padding uniformity as
Lemma~\ref{lem:reflection}.
\end{corollary}
\begin{proof}
For $0\le k\le j$, let $E_k$ be the event of $j$ disjoint arms with
$k$ open arms and $j-k$ closed arms.
Apply the site-inversion construction to $E_k$ and its color-reversed
copy. Then $\Prob(E_k)^2$ is bounded by the probability of $2j$
disjoint full-plane arms with exactly $j$ arms of each color, after
truncation to a common lattice annulus below the characteristic length.
Sum over the finitely many possible balanced color words.  For each word,
use \cite[Theorem~27]{Nolin} to compare with criticality and
\cite[Proposition~20]{Nolin} to compare the critical nonconstant word
with the alternating word.  Critical fixed-factor regularity absorbs
the radius changes.  Taking square roots and then summing over
$k=0,\ldots,j$ proves the result.
\end{proof}

Lemma~\ref{lem:reflection} gives \eqref{eq:halfarm} in both regimes with $\beta=1+\eta_6/2>1$, supplying the half-plane input in Theorem~\ref{thm:armcriterion}.

The pivotal scale and its boundary-arm compensation are recorded in \cite[Remark~36, equation~(7.17)]{Nolin}. We state the uniform interior estimates needed here.

\begin{lemma}[Bulk pivotal estimates]\label{lem:bulkmoments}
Fix a positive relative distance $\gamma$ from the sides of $Q$. Uniformly for relevant coordinates $x,y$ in this interior region,
\begin{align}
 c\alpha_4(R)\le\Prob(x\in\cP_f)&\le C\alpha_4(R),\label{eq:bulkone}\\
 \Prob(x,y\in\cP_f)&\le C\alpha_4(R)\alpha_4(|x-y|),\quad x\ne y.\label{eq:bulktwo}
\end{align}
The constants may depend on $\gamma$ and the previously fixed parameters. Consequently
\begin{equation}\label{eq:meancomparison}
 I(f)\asymp m(R),\qquad J(f)\asymp I(f).
\end{equation}
\end{lemma}
\begin{proof}
Pivotality at $x$ forces four alternating arms to radius $c_\gamma R$, for a sufficiently small fixed $c_\gamma>0$, giving the one-point upper bound. For the lower bound, require four arms from the coordinate to a square of radius $c_\gamma R$, with separated landings directed towards the four distinguished boundary arcs. Arm separation and comparison give probability at least $c\alpha_4(R)$. Extend them through four separated corridors to the respective sides of $Q$, open towards the left and right and closed towards the top and bottom. Uniform RSW and separated-arm extension bound the conditional probability of these extensions below by a positive constant, uniformly in the profile and scale. For mixed colors this gluing is justified by the separated-domain form of positive association, as in \cite[Lemma~13 and Theorem~11]{Nolin}; in the critical bond setting the corresponding arm-separation construction is reviewed in \cite[Section~2.2]{GPS}. The resulting outside configuration makes $x$ pivotal. Finite lattice endpoint adjustments are absorbed by finite energy and the fixed microscopic cutoff.

For the pair estimate put $d=|x-y|$. When $d$ is small compared with $R$, use two disjoint four-arm annuli of outer radius $cd$ about $x,y$, and a disjoint common four-arm annulus from radius $Cd$ to $c_\gamma R$. Independence and the full-plane arm comparison give
\[
 \Prob(x,y\in\cP_f)\le C\alpha_4(d)^2\alpha_4(d,R)
 \asymp C\alpha_4(d)\alpha_4(R).
\]
When $d\asymp R$, the two disjoint microscopic-to-macroscopic annuli give $C\alpha_4(R)^2$, the same bound by fixed-factor regularity. For bounded $d$, the one-point bound suffices. This proves \eqref{eq:bulktwo}.

There are at least $cR^2$ coordinates in a fixed interior rectangle. Summing \eqref{eq:bulkone} gives $I(f)\ge cm(R)$. Conversely, at terminal scale each occupied cell contains at most $M$ pivotal coordinates, so Proposition~\ref{prop:occupancy} gives $I(f)\le M\E N_{s_L}(\cP_f)\le Cm(R)$. Finally, $4p_0(1-p_0)I(f)\le J(f)\le I(f)$ by \eqref{eq:biasedinfluence}.
\end{proof}

\section{Entropy at each spatial resolution}\label{sec:mesoscopic}

\begin{lemma}[Two-coordinate identity]\label{lem:biasedpair}
For any Boolean function under independent coordinate parameters, define $\rho_i(1)=p_i$ and $\rho_i(0)=1-p_i$. Then, for $i\ne j$,
\begin{equation}\label{eq:pairidentity}
 \Prob(i,j\in\cS_f)=4\E\bigl[\rho_i(1-X_i)\rho_j(1-X_j)
                         \one_{\{i,j\in\cP_f\}}\bigr].
\end{equation}
In particular, if $p_i\in[p_0,1-p_0]$,
\begin{equation}\label{eq:paircompare}
 4p_0^2\Prob(i,j\in\cP_f)
 \le\Prob(i,j\in\cS_f)
 \le4(1-p_0)^2\Prob(i,j\in\cP_f).
\end{equation}
\end{lemma}
\begin{proof}
Fix all coordinates except $i,j$, and write $a_{uv}=f(X_i=u,X_j=v)$. The squared coefficient of the normalized two-coordinate character on this fiber is
\[
 p_i(1-p_i)p_j(1-p_j)
 (a_{11}-a_{10}-a_{01}+a_{00})^2.
\]
Let $n$ be the number of corners $(u,v)$ at which both coordinates are pivotal. For every Boolean two-coordinate table,
\[
 (a_{11}-a_{10}-a_{01}+a_{00})^2=4n.
\]
Indeed, the sixteen tables consist of constants and one-coordinate tables ($n=0$), tables with one exceptional corner ($n=1$), and the two parity tables ($n=4$); their squared differences are respectively zero, four, and sixteen. Each corner's probability multiplied by $\rho_i(1-u)\rho_j(1-v)$ is $p_i(1-p_i)p_j(1-p_j)$. Averaging this fiber identity over the outside coordinates gives \eqref{eq:pairidentity}, since Parseval identifies the averaged squared two-coordinate coefficient with $\sum_{S\supseteq\{i,j\}}\widehat f(S)^2$. The bounds on each $\rho$ give \eqref{eq:paircompare}.
\end{proof}

When $p_i=p_j=1/2$, \eqref{eq:pairidentity} gives equality of the two-point inclusion probabilities; see \cite[equation~(2.13)]{GPS} for the uniform product measure.

\begin{lemma}[Guards for several blocks]\label{lem:blockguards}
Let $f$ be Boolean under a product measure, and let $B_1,\ldots,B_k$ be disjoint coordinate sets. Suppose each event $E_i$ is measurable outside every $B_j$, makes $f$ independent of all coordinates in $B_i$, and has probability at least $q>0$. Suppose also that the events $E_i$ depend on disjoint sets of coordinates. Then, for $A=\cS_f$ or $\cP_f$,
\begin{equation}\label{eq:allhit}
 \Prob(A\cap B_i\ne\varnothing\text{ for every }i)\le(1-q)^k.
\end{equation}
\end{lemma}
\begin{proof}
Let $\E_{B_i}$ average over $B_i$, and set $D_i=\mathrm{Id}-\E_{B_i}$. These are commuting orthogonal projections. Write $D=\prod_iD_i$ and $M=\prod_i(1-\one_{E_i})$. Expanding $M$ shows $Df=D(fM)$: any nonempty expansion term contains an event $E_i$ on which the entire term is independent of $B_i$, so $D_i$ annihilates it. On the Fourier side $D$ retains exactly the supports meeting every $B_i$. Hence
\[
 \Prob(\cS_f\cap B_i\ne\varnothing\ \forall i)
 =\|Df\|_2^2=\|D(fM)\|_2^2
 \le\|fM\|_2^2
 =\Prob\Bigl(\bigcap_iE_i^c\Bigr)\le(1-q)^k.
\]
For pivotals, the all-hit event is contained in $\bigcap_iE_i^c$, giving the same bound directly.
\end{proof}

\begin{theorem}[Matching occupied-pattern entropy]\label{thm:mesoscopic}
In both regimes of Section~\ref{sec:inputs}, for every retained nonroot dyadic scale,
\[
 H(Z_s(A))\asymp\E N_s(A)\asymp m(R)/m(s),
 \qquad A=\cS_f\text{ or }\cP_f.
\]
The entropy comparison holds at the root for the pivotal set. For the spectral set its root entropy is \eqref{eq:rootexact}.
\end{theorem}
\begin{proof}
The entropy and occupancy upper bounds were proved in Section~\ref{sec:upper}. We first prove the lower bounds when $s\le R/(32\kappa)$. Translate geometric coordinates so that $Q=[0,a]\times[0,b]$; the resulting lattice phase is covered by the uniform estimates above. Retain only cells in the central rectangle
\[
 Q^{\mathrm{mid}}=[a/4,3a/4]\times[b/4,3b/4].
\]
At these dyadic levels their number is $(R/(2s))^2$. They contain no projected exterior centers.

For each such cell $C$ with center $z_C$, require a closed lattice circuit in the square annulus with radii $s$ and $2s$ about $z_C$. At sufficiently large fixed $s_*$ this guard lies outside $J_C$, uses only coordinates in $z_C+[-3s,3s]^2$, and encloses every tile of $J_C$. Uniform RSW and positive association give probability at least $q>0$, by joining long-direction closed crossings of the four overlapping strips forming the annulus. These are center-line closed-site circuits on the triangular lattice and closed dual-bond circuits on the square lattice. They are disjoint from open tile paths and make the crossing independent of $J_C$.

The enlarged boxes lie inside $Q$ by the fine-scale cutoff. Color the cell indices modulo an integer $b_\kappa>6\kappa$ in each direction. Centers of distinct cells of one color have $\ell^\infty$ separation greater than $6s$, so their enlarged boxes are disjoint. Each box contains both its block $J_C$ and its guard support, and the guard avoids its own block. Within one color, therefore, every guard is measurable outside every selected block, and the guard supports are disjoint. Lemma~\ref{lem:blockguards} applies.

Let $W_s(A)$ record the occupied central cells. For each fixed pattern $w$, one of the $b_\kappa^2$ colors contains at least $|w|/b_\kappa^2$ of its occupied cells. Apply \eqref{eq:allhit} to those cells. This gives the atom bound
\[
 \Prob(W_s=w)\le(1-q)^{|w|/b_\kappa^2}.
\]
The color may be chosen separately for each deterministic atom. Taking logarithmic reciprocals and averaging yields
\begin{equation}\label{eq:entropybulk}
 H(Z_s(A))\ge H(W_s(A))\ge c\E|W_s(A)|.
\end{equation}

Each central cell contains $\asymp_\kappa s^2$ coordinates, after increasing the fixed cutoff $s_*$ if necessary. Put $X_C=|A\cap J_C|$. Lemma~\ref{lem:bulkmoments}, the one-coordinate identity \eqref{eq:biasedinfluence}, and Lemma~\ref{lem:biasedpair} imply
\[
 \E X_C\ge cs^2\alpha_4(R),\qquad
 \E X_C^2\le Cs^2\alpha_4(R)
             \left(1+\sum_{t\le Cs\ \mathrm{dyadic}}t^2\alpha_4(t)\right).
\]
The second estimate counts neighbors of each coordinate in dyadic distance shells; the diagonal and bounded microscopic distances contribute the initial one. The growth bound \eqref{eq:massgrowth} and fixed-factor regularity give $1+\sum_{t\le Cs}t^2\alpha_4(t)\le Cm(s)$ for this dyadic sum. Cauchy--Schwarz therefore gives
\[
 \Prob(X_C>0)\ge\frac{(\E X_C)^2}{\E X_C^2}
 \ge c\frac{s^2\alpha_4(R)}{m(s)}.
\]
Summing over the central cells proves $\E|W_s(A)|\ge cm(R)/m(s)$. Equation~\eqref{eq:entropybulk} proves both fine-scale lower bounds.

There are only a bounded number, depending on $\kappa$, of remaining nonroot scales. Let $C_0$ be the upper-left cell at depth one and let $Y=\one_{\{A\cap J_{C_0}\ne\varnothing\}}$. Applying the preceding first- and second-moment argument to a fixed proportional square strictly inside $C_0$ gives $\Prob(Y=1)\ge a_0>0$. An open left--right crossing confined to a bottom strip of height $b/4$ has probability at least $b_0>0$ by uniform RSW. For sufficiently large $R_0$, its coordinate support, including the fixed tile padding, is disjoint from $J_{C_0}$. This event forces $f=1$ irrespective of that block, so it implies $Y=0$ for pivotals. For the spectral sample, \eqref{eq:blockvariance} gives $\Prob(Y=1)\le1-b_0$, because the conditional variance vanishes on the strip event. Thus $H(Y)\ge c>0$. Every nonroot dyadic pattern determines $Y$, while $m(R)/m(s)\asymp1$ at the remaining scales. This proves their entropy lower bounds; their occupancy lower bounds follow from $\Prob(Y=1)\ge a_0$.

At the pivotal root, nonemptiness has probability at least $a_0$. Two open left--right crossings in top and bottom strips force the pivotal set to be empty, since one path survives any single-coordinate flip. Their coordinate supports are disjoint for sufficiently large $R_0$, so uniform RSW and independence give positive probability for this event. Root pivotal occupancy consequently has entropy bounded below by a positive constant. For the spectral sample the empty atom has probability $\widehat f(\varnothing)^2=(\E f)^2$, proving \eqref{eq:rootexact}.
\end{proof}

\begin{proof}[Proof of Theorems~\ref{thm:main} and~\ref{thm:scales}]
Lemmas~\ref{lem:reflection} and \ref{lem:bulkmoments}, together with Theorem~\ref{thm:armcriterion}, prove the influence comparisons and the upper bounds. The microscopic guards of Corollary~\ref{cor:guardcross} give $H(\cS_f)\ge cJ(f)$. Theorem~\ref{thm:pivlower} gives $H(\cP_f)\ge h_2(p_0)I(f)-1$ in the triangular regime, and $I(f)-1$ in the critical bond regime. The lower growth of $m(R)$ and \eqref{eq:meancomparison} allow the one to be absorbed by increasing $R_0$. Finally, Theorem~\ref{thm:mesoscopic} supplies all the resolution statements.
\end{proof}

\begin{corollary}[Information between separated resolutions]\label{cor:resolution_information}
There is a fixed dyadic $K>1$ such that, for either random set in Theorem~\ref{thm:scales} and retained dyadic scales $s,t$ with $t\ge Ks$,
\[
 H(Z_s(A)\mid Z_t(A))\asymp\frac{m(R)}{m(s)}.
\]
Moreover, uniformly along admissible rectangles and retained scales $s=s(R)\to\infty$,
\[
 \frac{H(A\mid Z_s(A))}{H(A)}=1-O\bigl(m(s)^{-1}\bigr)\longrightarrow1.
\]
\end{corollary}
\begin{proof}
Nesting makes $Z_t$ a function of $Z_s$, so the first conditional entropy equals $H(Z_s)-H(Z_t)$. Let $c_1,C_1$ be the entropy comparison constants and $c_2,\eta>0$ the constants in \eqref{eq:massgrowth}. Choosing $K$ so that $C_1/(c_2K^\eta)\le c_1/2$ gives
\[
 H(Z_s)-H(Z_t)\ge\frac{c_1m(R)}{m(s)}-\frac{C_1m(R)}{m(t)}\ge\frac{c_1m(R)}{2m(s)}.
\]
The upper bound follows from $H(Z_s\mid Z_t)\le H(Z_s)$; the entropy upper estimate also applies when $t=R$. Finally, $Z_s$ is a function of $A$, whence $H(A\mid Z_s)=H(A)-H(Z_s)$. Divide by $H(A)\asymp m(R)$ and use Theorem~\ref{thm:scales} and the lower power growth of $m$.
\end{proof}

\section{Local geometry and spectral coercivity}\label{sec:localgeometry}

From this section onwards, lattice statements concern triangular-lattice site percolation unless stated otherwise.

\subsection{Pivotal sets contain no interior triangle}\label{sec:pivotal}

\begin{lemma}\label{lem:forest}
The subgraph induced by the full interior pivotal tiles is a linear forest: a disjoint union of paths, allowing isolated vertices. In particular, it contains no triangle.
\end{lemma}
\begin{proof}
Suppose first that the open left--right crossing occurs. Every pivotal is open and lies on every open crossing. Choose a simple crossing path. If three mutually adjacent interior sites on this path were pivotal, order their visits along the path. The edge joining the first to the third would replace the segment between them and avoid the middle site. All three tiles and the joining edge are inside the domain. The resulting open walk contains a crossing path avoiding that pivotal, a contradiction.

More generally, list all pivotal vertices in the order in which the chosen simple path visits them. An edge joining two nonconsecutive vertices in that list would give the same shortcut and avoid an intervening pivotal. Every edge between full interior pivotals therefore joins consecutive vertices in this order. Their induced graph is a subgraph of a path and hence a linear forest.

If the open left--right crossing fails, Hex duality gives a closed top--bottom crossing. This is the complementary event, so its pivotal set is the same. Apply the preceding argument to a simple closed crossing path. The same conclusion holds in either case.
\end{proof}

The spectral separator below uses only the triangle prohibition. No assertion about cut boundary tiles is required.

\subsection{A six-cycle coercivity inequality}\label{sec:coercivity}

Let $C_6$ have vertex set $\mathbb Z/6\mathbb Z$, with its cyclic edges, and let $\cI$ be its independent subsets. There are eighteen: the empty set, six singletons, nine pairs, and two alternating triples. On six fair signs $z=(z_0,\ldots,z_5)$, let $Q$ be the orthogonal Walsh projection onto
\[
 \operatorname{span}\{\chi_A:A\in\cI\}.
\]
Define
\[
 L=\left\{z:\sum_{i=0}^5\one_{\{z_i\ne z_{i+1}\}}\le2\right\},
 \qquad H=L^c.
\]
Each of $L$ and $H$ has thirty-two configurations. The configurations in $L$ have an empty, full, or single cyclic interval of plus signs.

\begin{lemma}\label{lem:rank}
If $q\in\operatorname{Ran}Q$ vanishes on $L$, then $q=0$.
\end{lemma}
\begin{proof}
Put $y_i=(1+z_i)/2$. Since the family $\cI$ is closed under subsets, the span of its Walsh characters equals the span of the monomials $y_A=\prod_{i\in A}y_i$, $A\in\cI$. Write $q=\sum_{A\in\cI}b_Ay_A$.

Evaluate first at the empty open set, and then at each singleton open set. This gives $b_\varnothing=0$ and all singleton coefficients zero. A three-site cyclic interval contains exactly one independent pair, its two endpoints; evaluating on all six such intervals kills the six distance-two pair coefficients. A four-site interval then kills its opposite pair, since the other independent coefficients in that interval have already vanished. Three such intervals kill the remaining three pairs. Finally, a five-site interval contains exactly one of the two alternating triples. Intervals omitting a site of each parity kill both triple coefficients. All evaluations were on $L$, so $q=0$.
\end{proof}

\begin{proposition}\label{prop:coercivity}
Every function $g$ on six signs which vanishes on $L$ satisfies
\begin{equation}\label{eq:coercivity}
 \|(\Id-Q)g\|_2^2\ge\astar\|g\|_2^2.
\end{equation}
The constant in \eqref{eq:astar} is optimal on this function space.
\end{proposition}
\begin{proof}
Lemma~\ref{lem:rank} already implies the inequality with some positive constant by compactness of the unit sphere of functions supported on $H$. We compute the best constant.

Let $M$ be the $32\times18$ matrix
\[
 M_{z,A}=\chi_A(z),\qquad z\in L,\quad A\in\cI,
\]
and put $G=M^{\mathsf T}M$. If $q=\sum_{A\in\cI}u_A\chi_A$, then
\begin{equation}\label{eq:gramnorm}
 \|q\|_2^2=\|u\|^2,\qquad
 \|\one_Lq\|_2^2=\frac1{64}u^{\mathsf T}Gu.
\end{equation}
We record the small-block calculation to specify the constant without numerical approximation.

The symmetry $z\mapsto-z$ separates even and odd degrees. Order the even block by the constant and the six distance-two pairs, followed by the three opposite pairs. Then
\[
 G_{\rm even}=32I_{10}+
 \begin{pmatrix}0&K\\K^{\mathsf T}&0\end{pmatrix},
 \qquad K^{\mathsf T}K=512I_3+112J_3.
\]
Here $J_3$ is the all-ones matrix. The first row of $K$ is $(-4,-4,-4)$; the other rows are two copies of each permutation of $(-4,12,12)$. The eigenvalues of this block are
\[
 32\pm4\sqrt{53},\quad
 32\pm16\sqrt2\ \text{(each twice)},\quad
 32\ \text{(four times)}.
\]

In the odd block, the singleton matrix is circulant with first row $(32,12,0,-4,0,12)$; its eigenvalues are $52,48,16,12,16,48$. The alternating triples have block
\[
 \begin{pmatrix}32&-4\\-4&32\end{pmatrix}.
\]
A singleton couples with coefficient zero to its own parity triple and with coefficient $-4$ to the other. Only the constant and alternating singleton modes couple to these triples, giving the two blocks
\[
 \begin{pmatrix}52&-4\sqrt3\\-4\sqrt3&28\end{pmatrix},
 \qquad
 \begin{pmatrix}12&4\sqrt3\\4\sqrt3&36\end{pmatrix}.
\]
Thus the odd eigenvalues are $40\pm8\sqrt3$, $24\pm8\sqrt3$, and $48,16$ each twice. The least eigenvalue of $G$ is $32-4\sqrt{53}$.

By \eqref{eq:gramnorm}, $\|\one_Lq\|_2^2\ge\astar\|q\|_2^2$ on $\operatorname{Ran}Q$. If $g=\one_Hg$, then
\[
 |\langle g,q\rangle|=|\langle g,\one_Hq\rangle|
 \le\sqrt{1-\astar}\,\|g\|_2\|q\|_2.
\]
Take $q=Qg$ and use Pythagoras to obtain \eqref{eq:coercivity}. For sharpness, choose a least-eigenvalue vector $q$ in \eqref{eq:gramnorm} and put $g=\one_Hq$. Then $Qg=(1-\astar)q$, and equality follows.
\end{proof}

\section{Exterior conditioning and spectral face density}\label{sec:faces}
\subsection{Preservation under exterior operators}\label{sec:conditioning}

Consider a finite site-connection graph with a distinguished nonterminal site $v$ whose only neighbors are $v_0,\ldots,v_5$, with all cyclic ring edges present. Let $W_v$ be this seven-site wheel. Let $f$ be the sign of the event connecting two fixed terminal sets. For a site $v$, define the sign derivative
\[
 \partial_v f(\omega_{v^c})
 =\E_v[\omega_vf]
 =\frac{f(1,\omega_{v^c})-f(-1,\omega_{v^c})}{2}.
\]

\begin{lemma}\label{lem:zero}
For every exterior configuration, $\partial_vf$ vanishes when the six ring signs lie in $L$.
\end{lemma}
\begin{proof}
In such a configuration, all open neighbors of $v$ are joined by an open ring path, unless there are no open neighbors. Any terminal path through the open center can therefore replace its visit to $v$ by an open ring path. The resulting walk contains a terminal path avoiding $v$. If there are no open neighbors, opening $v$ cannot create a terminal connection. Thus reversing $v$ never changes the event.
\end{proof}

Let $E_v$ average over the center sign, let $P_v=\Id-E_v$, and extend the ring projection $Q_v$ by the identity on all other coordinates. Set
\[
 T_v=P_v(\Id-Q_v).
\]
This is the Fourier projection onto supports containing $v$ and at least one pair of cyclically adjacent neighbors. Such supports contain at least one of the six elementary faces incident to $v$.

\begin{proposition}\label{prop:outside}
If $A$ is any linear operator acting only on coordinates outside $W_v$, then
\begin{equation}\label{eq:outside}
 \|T_vAf\|_2^2\ge\astar\|P_vAf\|_2^2.
\end{equation}
Consequently, for every spectral event $\mathcal H$ determined by $\cS_f\cap W_v^c$,
\begin{equation}\label{eq:conditional}
 \Prob\bigl(T_v(\cS_f)=1,\mathcal H\bigr)
 \ge\astar\Prob\bigl(v\in\cS_f,\mathcal H\bigr),
\end{equation}
where $T_v(S)=1$ means that $S$ contains an incident face at $v$.
\end{proposition}
\begin{proof}
The zero in Lemma~\ref{lem:zero} holds as an identity in every exterior coordinate. An operator on those coordinates preserves it. Hence $\partial_vAf$ still vanishes on $L$, fiber by fiber. Since
\[
 P_vAf=\omega_v\partial_vAf,
 \qquad T_vAf=\omega_v(\Id-Q_v)\partial_vAf,
\]
Proposition~\ref{prop:coercivity}, followed by integration over the exterior coordinates, gives \eqref{eq:outside}.

For \eqref{eq:conditional}, take $A$ to be the diagonal Walsh projector retaining exactly the exterior supports satisfying $\mathcal H$. Parseval identifies the two squared norms with the two probabilities. The function $Af$ need not be Boolean, monotone, or nonnegative.
\end{proof}

The whole wheel must remain unconditioned in \eqref{eq:conditional}. An exterior event for a single chosen face may involve the other ring coordinates and does not satisfy this hypothesis. The estimate concerns the union of the six incident-face events.

\subsection{Conditional Bernoulli bounds and face density}\label{sec:amplification}

Take a family of disjoint full wheels $W_{v_1},\ldots,W_{v_k}$. For the spectral sample define
\[
 x_j=\one_{\{v_j\in\cS_f\}},\qquad
 y_j=\one_{\{T_{v_j}(\cS_f)=1\}},
 \qquad X=\sum_jx_j,\quad Y=\sum_jy_j.
\]
We have $y_j\le x_j$.

\begin{lemma}\label{lem:bernoulli}
Conditional on the complete center-inclusion pattern $(x_1,\ldots,x_k)$, the variable $Y$ stochastically dominates a binomial random variable with $X$ trials and success probability $\astar$. In particular, for $t\ge0$ and $b(t)=1-\astar+\astar e^{-t}$,
\begin{equation}\label{eq:classmgf}
 \E\bigl[e^{-tY}b(t)^{-X}\bigr]\le1.
\end{equation}
\end{lemma}
\begin{proof}
Fix a center pattern having positive probability and list its occupied centers. Reveal the corresponding $y_j$ in any deterministic order. At a currently occupied center $v_j$, the specified values of all other $x_i$, and every preceding success or failure of $y_i$, depend only on spectral coordinates outside $W_{v_j}$. This uses disjointness of the wheels. Proposition~\ref{prop:outside}, divided by the probability of the conditioning event together with $x_j=1$, gives
\[
 \Prob(y_j=1\mid x_1,\ldots,x_k,\text{preceding }y_i)\ge\astar
\]
whenever this conditional probability is defined. Successive coupling with independent uniform random variables gives the stochastic domination. Equivalently, iterated conditional expectation gives
\[
 \E[e^{-tY}\mid x_1,\ldots,x_k]\le b(t)^X,
\]
which proves \eqref{eq:classmgf}.
\end{proof}

\begin{proof}[Proof of Theorem~\ref{thm:finite}]
Use axial coordinates $(i,j)$ for the triangular lattice, with neighbor differences $\pm(1,0)$, $\pm(0,1)$, and $\pm(1,-1)$. Color the interior centers by $i+3j\pmod7$. The seven vertices of any wheel have all seven colors: the neighbor increments have residues $\pm1,\pm2,\pm3$. Any two sites at graph distance at most two lie in a common wheel. Thus distinct centers of the same color have distance at least three, and their wheels are disjoint. Write $X_c,Y_c$ for the preceding counts in color $c$, $1\le c\le7$, and put
\[
 X=\sum_cX_c=X_R,\qquad Y=\sum_cY_c.
\]
Each successful center has an incident face. Each face has at most three vertices, so
\begin{equation}\label{eq:Yfaces}
 Y\le3K_R(\cS_R).
\end{equation}
Apply H\"older's inequality to the seventh roots of the seven nonnegative variables in \eqref{eq:classmgf}. It gives
\begin{equation}\label{eq:globalmgf}
 \E\bigl[e^{-tY/7}b(t)^{-X/7}\bigr]\le1.
\end{equation}
No independence between color classes is used.

On $\{Y\le\theta X,\ X\ge m\}$, the random variable in \eqref{eq:globalmgf} is at least
\[
 \exp\left\{\frac{m}{7}\bigl(-t\theta-\ln b(t)\bigr)\right\}
\]
whenever the quantity in parentheses is positive. Markov's inequality and optimization over $t\ge0$ yield
\[
 \Prob(Y\le\theta X,\ X\ge m)
 \le e^{-mD(\theta\Vert\astar)/7}.
\]
Use \eqref{eq:Yfaces} to obtain \eqref{eq:density}. Taking $\theta=0$, or letting $t\to\infty$ in \eqref{eq:globalmgf}, proves \eqref{eq:noface}.
\end{proof}

\begin{lemma}\label{lem:fibers}
Put $q=1-\astar$, $F_R=\{K_R(\cS_R)=0\}$, and $E_R=\{\cS_R\subseteq V_R^\circ\}$. For every integer $n\ge1$,
\begin{equation}\label{eq:fibers}
 \Prob(F_R,\ |\cS_R|=n,\ E_R)
 \le7q^{\lceil n/7\rceil}\Prob(1\le|\cS_R|\le7n).
\end{equation}
\end{lemma}
\begin{proof}
Fix a color class $V_c$ and partition all supports $S$ by
\[
 J=S\cap V_c,\qquad W_J=\bigcup_{v\in J}W_v,
 \qquad T=S\setminus W_J.
\]
In such a fiber, the $k=|J|$ active centers are included and the entire support outside their disjoint wheels is fixed. Inactive centers of the same color lie outside $W_J$, so their absence is already specified by $T$. The conditional argument in Lemma~\ref{lem:bernoulli} remains valid with this extra exterior condition: every preceding marker history and every other center condition is exterior to the current wheel. The probability that all $k$ markers fail is therefore at most $q^k$ in every positive-probability fiber.

Write $\ell=|T|$. Every support in the fiber has size between $\ell+k$ and $\ell+7k$. If $k\ge\lceil n/7\rceil$ and $\ell+k\le n$, the whole fiber lies in $\{1\le|S|\le7n\}$. Sum the conditional failure bound over these disjoint fibers. This gives
\[
 \Prob(F_R,\ |\cS_R|=n,\ X_c\ge\lceil n/7\rceil)
 \le q^{\lceil n/7\rceil}\Prob(1\le|\cS_R|\le7n).
\]
On $E_R\cap\{|\cS_R|=n\}$, the seven center counts sum to $n$, so at least one is at least $\lceil n/7\rceil$. Sum over colors to prove \eqref{eq:fibers}.
\end{proof}

The conditioning in this proof fixes complete fibers, not the total Fourier degree. Other members of a selected fiber may meet the boundary collar; this is why the probability on the right of \eqref{eq:fibers} is unrestricted.

\begin{corollary}\label{cor:robust}
For critical square crossings,
\[
 \Prob\left(K_R(\cS_R)\ge\frac{\astar}{6}X_R\right)\longrightarrow1.
\]
Moreover, with probability tending to one, at least $\astar X_R/36$ sites must be deleted from $\cS_R$ to obtain a set containing no full interior elementary face.
\end{corollary}
\begin{proof}
Section~\ref{sec:critical} proves that $X_R\to\infty$ in probability. Apply \eqref{eq:density} with $\theta=\astar/2$. Each site belongs to at most six elementary faces, so deleting $k$ sites can destroy at most $6k$ of the original faces. The second assertion follows.
\end{proof}

\section{Critical singularity and quantitative bounds}\label{sec:criticalrates}
\subsection{Critical square estimates}\label{sec:critical}

Fix an admissible microscopic scale $r_0$ large enough for all four- and six-arm events, and write $\alpha_j(R)=\alpha_j(r_0,R)$. Fixed changes of $r_0$ change these probabilities only by constant factors. The triangular-lattice four-arm exponent is
\begin{equation}\label{eq:a4}
 \alpha_4(R)=R^{-5/4+o(1)}.
\end{equation}
We use the square spectral lower-tail theorem of \cite[Theorem~1.1 and equation~(1.7)]{GPS}, together with \eqref{eq:a4}; the arm exponent is due to \cite{SW}. These imply, for every fixed $A>0$,
\begin{equation}\label{eq:small}
 \Prob(0<|\cS_R|\le A\ln R)\le R^{-1/2+o(1)}.
\end{equation}
Indeed, $\E|\cS_R|\asymp R^2\alpha_4(R)=R^{3/4+o(1)}$, and the lower-tail exponent in the normalized spectral size is $2/3+o(1)$. Substituting $A\ln R$ gives \eqref{eq:small}.

We spell out the boundary estimate, since the local algebra only applies to complete interior wheels.

\begin{lemma}\label{lem:boundary}
Let $B_R=V_R\setminus V_R^\circ$ be a collar of fixed lattice width around the square boundary. Then
\[
 \Prob(\cS_R\cap B_R\ne\varnothing)
 \le\E|\cS_R\cap B_R|\le C R^{-1}.
\]
\end{lemma}
\begin{proof}
By \eqref{eq:biasedinfluence}, at the fair product measure the one-point probabilities agree:
\[
 \Prob(v\in\cS_R)=\Prob(v\in\cP_R).
\]
Group collar sites by their dyadic distance $d\ge1$ to the nearest corner. Each group has $O(d)$ sites. Pivotality requires three half-plane arms from fixed scale to $d$, and two quarter-plane arms from scale $d$ to $R$. Use constant buffers to make the annuli disjoint. The side and corner estimates in \cite[equations~(4.12)--(4.13) and Section~7.2]{GPS} bound this probability by
\[
 C d^{-2}(d/R)^{1+\zeta}
\]
for some $\zeta>0$. At $d\asymp1$ or $d\asymp R$, omit the corresponding bounded-ratio annulus. Summation gives
\[
 \E|\cS_R\cap B_R|
 \le C\sum_{d\le R\text{ dyadic}}d\,d^{-2}(d/R)^{1+\zeta}
 \le C R^{-1}.
\]
Fixed changes in collar width only change the constant.
\end{proof}

The conformal crossing limit for a square gives $\E f_R\to0$, so $\epsilon_R\to0$; see \cite{Smirnov}. This limit, rather than an assumed finite-lattice symmetry, is enough here. Equations \eqref{eq:small} and Lemma~\ref{lem:boundary} imply $X_R\to\infty$ in probability.

For the constant-factor estimate, we use the more detailed form of \cite[Proposition~4.1]{GPS}. With $\gamma_4(R)=R^2\alpha_4(R)^2$, that proposition bounds the probability of exactly $k$ occupied boxes of the fixed scale $r_0$ by $C g(k)\gamma_4(R)$, where $g$ can be chosen increasing and
\[
 g(k)\le\exp\{C\ln^2(k+2)\}.
\]
Each site or tile meets at most a fixed number $C_0$ of these boxes. Absorbing this bounded overlap into the constants and into the increasing subexponential function $g$, we obtain
\begin{equation}\label{eq:gpssize}
 \Prob(1\le|\cS_R|\le7n)\le Cn g(7n)\gamma_4(R).
\end{equation}
Combine this with Lemma~\ref{lem:fibers}. Since $\ln(ng(7n))=o(n)$, there are constants $C,\lambda>0$ such that
\begin{equation}\label{eq:sizeresolved}
 \Prob(F_R,\ |\cS_R|=n,\ E_R)
 \le C\gamma_4(R)e^{-\lambda n},\qquad n\ge1.
\end{equation}

\begin{proof}[Proof of Theorems~\ref{thm:discrete} and \ref{thm:rate}]
Sum \eqref{eq:sizeresolved} over $n\ge1$ and add the collar probability from Lemma~\ref{lem:boundary}. Since $R^{-1}=o(\gamma_4(R))$,
\[
 \Prob(K_R(\cS_R)=0,\ \cS_R\ne\varnothing)\le C\gamma_4(R).
\]
For the reverse inequality, retain the singleton supports in a fixed bulk subsquare. Monotonicity gives
\[
 \widehat f_R(\{v\})=\E\partial_vf_R
 =\Prob(v\in\cP_R)\asymp\alpha_4(R)
\]
uniformly for these sites, by the bulk four-arm comparison. There are order $R^2$ such sites. A singleton contains no triangular face, so
\[
 \Prob(K_R(\cS_R)=0,\ \cS_R\ne\varnothing)
 \ge c R^2\alpha_4(R)^2.
\]
This proves \eqref{eq:rate}. Adding the empty mass and using $\epsilon_R\to0$ proves the spectral assertion in Theorem~\ref{thm:discrete}. Lemma~\ref{lem:forest} gives the pivotal assertion. The event $\{K_R>0\}$ then proves convergence of total variation to one.
\end{proof}

\begin{corollary}\label{cor:conditionalsize}
There are constants $C,\lambda>0$ such that
\[
 \Prob(|\cS_R|\ge m\mid F_R,\ \cS_R\ne\varnothing)
 \le C e^{-\lambda m}+O\bigl(R^{-1}/\gamma_4(R)\bigr).
\]
In particular, these conditional size laws are asymptotically tight.
\end{corollary}
\begin{proof}
Sum \eqref{eq:sizeresolved} over $n\ge m$, add the collar bound, and divide by the lower bound in Theorem~\ref{thm:rate}.
\end{proof}

\subsection{The sharp cost of low face density}\label{sec:lowdensity}

The same rare-event scale governs every fixed face density below the local threshold. This strengthens both the triangle-free estimate and the qualitative density conclusion.

\begin{theorem}\label{thm:lowdensity}
For critical triangular-site square crossings and each fixed $0\le\delta<\astar/3$, put
\[
 F_R^\delta=\{K_R(\cS_R)\le\delta|\cS_R|\},\qquad
 \gamma_4(R)=R^2\alpha_4(R)^2.
\]
There are $C_\delta,\lambda_\delta>0$ such that, for every integer $n\ge1$,
\begin{equation}\label{eq:lowdensity_size}
 \Prob(F_R^\delta,\ |\cS_R|=n,\ E_R)
 \le C_\delta\gamma_4(R)e^{-\lambda_\delta n},
 \qquad E_R=\{\cS_R\subseteq V_R^\circ\}.
\end{equation}
Consequently,
\begin{equation}\label{eq:lowdensity_mass}
 \Prob(F_R^\delta,\ \cS_R\ne\varnothing)
 \asymp_\delta\gamma_4(R)=R^{-1/2+o(1)}.
\end{equation}
Moreover, for every integer $m\ge1$,
\[
 \Prob\bigl(|\cS_R|\ge m\mid F_R^\delta,\ \cS_R\ne\varnothing\bigr)
 \le C_\delta e^{-\lambda_\delta m}+C_\delta R^{-1}/\gamma_4(R).
\]
\end{theorem}
\begin{proof}
Choose $\theta$ with $3\delta<\theta<\astar$, and put
\[
 \eta=\frac{1-3\delta/\theta}{14}>0,\qquad
 d_\theta=D(\theta\Vert\astar)>0.
\]
On $E_R\cap\{|\cS_R|=n\}$ the seven center counts satisfy $\sum_cX_c=n$. On $F_R^\delta$, equation~\eqref{eq:Yfaces} gives $\sum_cY_c\le3\delta n$. There must therefore be a color $c$ such that
\begin{equation}\label{eq:lowdensity_color}
 X_c\ge\eta n,\qquad Y_c\le\theta X_c.
\end{equation}
Indeed, the classes with $X_c<\eta n$ contain fewer than $7\eta n$ centers in total. If every other class had $Y_c>\theta X_c$, then
\[
 \sum_cY_c>\theta(1-7\eta)n=\frac12(\theta+3\delta)n>3\delta n,
\]
a contradiction.

For each color, use the disjoint fibers $(J,T)$ from the proof of Lemma~\ref{lem:fibers}. Write $k=|J|$ and $\ell=|T|$. The fiber fixes the included centers and the entire support outside their disjoint wheels. Thus the sequential conditional argument of Lemma~\ref{lem:bernoulli} applies within the fiber, and its Chernoff bound is
\[
 \Prob(Y_c\le\theta k\mid J,T)\le e^{-d_\theta k}
\]
for every positive-probability fiber. No conditioning on total support size is used in this inequality. Only fibers with $k\ge\eta n$ and $\ell+k\le n$ can contribute to \eqref{eq:lowdensity_color} at size $n$. Such a fiber has $k\ge1$, and every one of its supports has size between $\ell+k$ and $\ell+7k\le7n$. Summing over these fibers and over the seven colors gives
\[
 \Prob(F_R^\delta,\ |\cS_R|=n,\ E_R)
 \le7e^{-\eta d_\theta n}\Prob(1\le|\cS_R|\le7n).
\]
By \eqref{eq:gpssize}, the right-hand side is at most $C\gamma_4(R)ng(7n)e^{-\eta d_\theta n}$. Since $\ln(ng(7n))=o(n)$, this proves \eqref{eq:lowdensity_size}, for instance with $\lambda_\delta=\eta d_\theta/2$ after increasing $C_\delta$ to cover the finitely many smaller $n$.

Sum over $n\ge1$ and use Lemma~\ref{lem:boundary}; its error $O(R^{-1})$ is $o(\gamma_4(R))$. This proves the upper bound in \eqref{eq:lowdensity_mass}. Every singleton belongs to $F_R^\delta$, including when $\delta=0$, so the bulk singleton argument in the proof of Theorem~\ref{thm:rate} gives its lower bound. Summing \eqref{eq:lowdensity_size} over $n\ge m$, adding the collar error, and dividing by that lower bound proves the conditional estimate.
\end{proof}

\subsection{The total variation deficit}\label{sec:TV}

Write $\mu_R$ and $\nu_R$ for the discrete spectral and pivotal laws. Their overlap is
\begin{equation}\label{eq:overlap}
 1-\TV(\mu_R,\nu_R)=\sum_A\min\{\mu_R(A),\nu_R(A)\}.
\end{equation}
We first construct shared mass on identical singleton supports.

\begin{lemma}\label{lem:singlepiv}
Uniformly for sites $v$ in a fixed bulk subsquare,
\[
 \Prob(\cP_R=\{v\})\ge c\alpha_6(R),
\]
where $\alpha_6(R)$ denotes the six-arm probability from fixed microscopic scale to $R$.
\end{lemma}
\begin{proof}
Prescribe six vertex-disjoint arms from a fixed microscopic neighborhood of $v$ to separated boundary intervals, in cyclic order: two open arms to the left, one closed arm to the top, two open arms to the right, and one closed arm to the bottom. The color word is $OOCOOC$. Polychromatic color-order comparison and arm separation give probability at least $c\alpha_6(R)$ for this event with prescribed inner and outer landing intervals; see \cite[Theorem~11, equation~(4.8), Proposition~12, and Proposition~20, equation~(5.2)]{Nolin}. Use the fenced, well-separated version of the arm event. Extensions to the six boundary intervals take place in disjoint corridors of fixed aspect ratio, each with one prescribed color; the locally monotone gluing estimate in \cite[Proposition~12]{Nolin} changes the probability by a constant factor.

Inside the fixed microscopic neighborhood, prescribe six disjoint corridors joining the arms in cyclic order to the six neighbors of $v$, colored $OOCOOC$, and set $v$ open. Choose the neighborhood large enough for these finitely many corridors. This fixed assignment costs a positive constant, independent of $R$ and the bulk site. The prescribed inner landing intervals allow it to be joined to the six exterior arms.

The four open arms form two left--right paths whose only common random site is $v$. Closing $v$ joins the two closed arms to a top--bottom closed path, so $v$ is pivotal. Every other site is avoided by at least one of the two original open crossings and cannot be pivotal. Thus the constructed event implies $\cP_R=\{v\}$.
\end{proof}

\begin{proof}[Proof of Theorem~\ref{thm:TVrate}]
Let $O_R$ be the nonempty part of the sum in \eqref{eq:overlap}. By the pivotal triangle prohibition, for any integer $m\ge1$,
\[
 O_R\le\nu_R(1\le|\cP_R|\le m)
       +\mu_R(F_R,\ |\cS_R|>m).
\]
The pivotal analogue of the fixed-scale box bound, given in \cite[Remark~4.6]{GPS}, yields
\[
 \nu_R(1\le|\cP_R|\le m)\le C m g(m)R^2\alpha_6(R),
\]
where $g(m)\le\exp\{C\ln^2(m+2)\}$. The second term is at most $CR^{-1}+q^{m/7}$ by the collar estimate and \eqref{eq:noface}. Take $m=\lceil A\ln R\rceil$ with $A$ sufficiently large. Since
\[
 R^2\alpha_6(R)=R^{-11/12+o(1)},
\]
we obtain $O_R\le R^{-11/12+o(1)}$.

For the lower bound, at every site in a fixed bulk subsquare,
\[
 \mu_R(\{v\})\ge c\alpha_4(R)^2,\qquad
 \nu_R(\{v\})\ge c\alpha_6(R)
\]
by the singleton spectral identity and Lemma~\ref{lem:singlepiv}. The ratio $\alpha_4(R)^2/\alpha_6(R)=R^{5/12+o(1)}$ tends to infinity. Summing the minima over order $R^2$ bulk sites gives
\[
 O_R\ge cR^2\alpha_6(R)=R^{-11/12+o(1)}.
\]

Finally, $\nu_R(\varnothing)$ is bounded below uniformly: require open crossings in two disjoint horizontal strips of fixed aspect ratio. RSW and independence give positive probability, and two disjoint crossings preclude a pivotal. Since $\mu_R(\varnothing)=\epsilon_R\to0$, the empty contribution to \eqref{eq:overlap} is exactly $\epsilon_R$ for all sufficiently large $R$. This proves the theorem.
\end{proof}

\section{Local criteria and scope extensions}\label{sec:localscope}
\subsection{An abstract local criterion}\label{sec:criterion}

The preceding proof separates a finite-dimensional condition from spectral-size estimates. The same argument gives a useful criterion for other connection functions.

\begin{proposition}\label{prop:criterion}
Suppose a sequence of normalized real functions $f_n$ on finite product cubes has tested centers $V_n^\circ$. For each center $v$, suppose there are a finite block $W_v\ni v$ and a local diagonal Fourier projection $T_v\le P_v$ onto an event of the support in $W_v$, with a uniform constant $0<a\le1$, such that
\[
 \|T_vAf_n\|_2^2\ge a\|P_vAf_n\|_2^2
\]
for every exterior linear operator $A$. Assume the blocks can be partitioned into $C$ disjoint families, where $C$ is independent of $n$. Let $Y_n$ count successful local tests and let $X_n=|\cS_{f_n}\cap V_n^\circ|$. Then
\[
 \Prob(Y_n=0,\ X_n\ge m)\le(1-a)^{m/C}.
\]
If $X_n\to\infty$ in probability and comparison laws assign probability zero to every successful test, the two sequences of laws are asymptotically mutually singular.
\end{proposition}
\begin{proof}
The conditional Bernoulli argument of Lemma~\ref{lem:bernoulli} applies within each family. H\"older across $C$ families gives the analogue of \eqref{eq:globalmgf}. Letting $t\to\infty$ yields the bound. The event $\{Y_n>0\}$ gives the final assertion.
\end{proof}

In particular, noise sensitivity suffices for the small-cardinality part of this criterion when the empty Fourier mass vanishes. For each fixed $k$ and fixed $0<\rho<1$, noise sensitivity implies
\[
 \rho^k\Prob(0<|\cS_{f_n}|\le k)
 \le\sum_{A\ne\varnothing}\rho^{|A|}\widehat f_n(A)^2\longrightarrow0.
\]
The sharper percolation lower-tail estimate is needed only for the decay exponent in Theorem~\ref{thm:rate}. The Fourier interpretation of noise sensitivity originates in \cite{BKS}; the quantitative percolation estimates used here are from \cite{GPS}.

\subsection{Product measures and general quads}\label{sec:extensions}

The local estimate persists for nonuniform product measures. Let $\eta_i\in\{0,1\}$ be independent with parameters $0<p_i<1$, and use the orthonormal basis
\[
 \psi_A=\prod_{i\in A}\frac{\eta_i-p_i}{\sqrt{p_i(1-p_i)}}.
\]
Spectral probabilities are the squared coefficients in this basis, for functions of unit $L^2$ norm.

\begin{proposition}\label{prop:biased}
Suppose every tested ring has its six parameters in $[\delta,1-\delta]$, where $0<\delta\le1/2$. For a site-connection sign function, Proposition~\ref{prop:outside}, Theorem~\ref{thm:finite}, and Lemma~\ref{lem:fibers} hold for the product spectral law with $\astar$ replaced by
\[
 a_\delta=\astar\left(\frac{\delta}{1-\delta}\right)^6.
\]
When the connection function is nonconstant, they also hold for its centered, variance-normalized version, whose spectral law is the original spectral law conditioned to be nonempty.
\end{proposition}
\begin{proof}
Because $\cI$ is downward closed, the span $V$ of $\{\psi_A:A\in\cI\}$ on a ring equals the span of its independent-set monomials. Thus $V$ is the same vector space as in Lemma~\ref{lem:rank}. The ring product measure has density, relative to the fair measure, between
\[
 m_p=2^6\prod_i\min(p_i,1-p_i),\qquad
 M_p=2^6\prod_i\max(p_i,1-p_i).
\]
For $g$ vanishing on $L$, norm comparison gives
\[
 \dist_p(g,V)^2
 \ge m_p\dist_{1/2}(g,V)^2
 \ge m_p\astar\|g\|_{1/2}^2
 \ge (m_p/M_p)\astar\|g\|_p^2
 \ge a_\delta\|g\|_p^2.
\]
The product derivative is
\[
 D_v^pf=\E_v[\psi_vf]
 =\sqrt{p_v(1-p_v)}\,[f(1,\cdot)-f(0,\cdot)].
\]
It vanishes on $L$ by Lemma~\ref{lem:zero}, and that identity survives every exterior operator. Since $P_vAf=\psi_vD_v^pAf$ and $\|\psi_v\|_p=1$, the preceding inequality proves the product version of Proposition~\ref{prop:outside}. All subsequent conditioning and coloring arguments are unchanged. Finally, subtracting the mean changes only the empty Fourier coefficient, and division by the standard deviation normalizes the remaining squared coefficients. The derivative inequalities are homogeneous, so the same constants apply after this normalization.
\end{proof}

Uniformly bounded product parameters alone do not imply escape of the spectral size or negligible boundary mass. Proposition~\ref{prop:biased} is a finite-volume result. For the uncentered Boolean sign function, the fair one-point identity becomes
\[
 \Prob_p(v\in\cS_f)=4p_v(1-p_v)\Prob_p(v\in\cP_f),
\]
as follows by taking the squared norm of $D_v^pf$. For its centered spectrum, divide the right-hand side by $\operatorname{Var}_p(f)$. Boundary estimates must use this normalization.

\begin{theorem}\label{thm:quads}
Let $Q$ be a fixed Jordan quad in the sense of \cite[Section~2.1]{GPS}, with its continuous hexagonal-tile crossing convention. Let $\mu_R$ and $\nu_R$ be the spectral and pivotal laws of the critical triangular-site crossing of $RQ$, and write
\[
 \mu_R^+=\mu_R(\,\cdot\mid S\ne\varnothing),\qquad
 \epsilon_R=\mu_R(\varnothing).
\]
Count only elementary faces consisting of full interior tiles. Then
\[
 \mu_R(K_R=0,\ S\ne\varnothing)\longrightarrow0,
 \qquad \TV(\mu_R^+,\nu_R)\longrightarrow1.
\]
The latter conclusion also holds with the pivotal law conditioned to be nonempty. Moreover, for every fixed $0<\delta<\astar/3$,
\[
 \mu_R^+\bigl(K_R(S)\ge\delta|S|\bigr)\longrightarrow1,
 \qquad
 1-\TV(\mu_R,\nu_R)
 =\min\{\epsilon_R,\nu_R(\varnothing)\}+o(1).
\]
\end{theorem}
\begin{proof}
By \cite[Theorem~7.4]{GPS}, nonzero spectral size is tight above and away from zero on the scale $R^2\alpha_4(R)$, which tends to infinity. Thus, for each fixed $m$,
\[
 \mu_R(0<|S|<m)\longrightarrow0.
\]
Equations~(7.7)--(7.8) in the proof of that theorem show that, for every $\eta>0$, there is a fixed $U'\Subset Q^\circ$ with
\[
 \limsup_{R\to\infty}\mu_R(S\not\subset RU')\le\eta.
\]
For large $R$, all sites in $RU'$ are tested centers. Consequently, for the full-wheel center set $V_R^\circ$,
\[
 \mu_R(S\not\subset V_R^\circ)\longrightarrow0.
\]
This interior approximation avoids any regularity assumption on the Jordan boundary. The face-free estimate now gives
\[
 \mu_R(K_R=0,\ S\ne\varnothing)
 \le\mu_R(S\not\subset V_R^\circ)
   +\mu_R(0<|S|<m)+(1-\astar)^{m/7}.
\]
Send $R\to\infty$, then $m\to\infty$.

For a fixed nonempty bulk set, the first- and second-moment estimates in \cite[Lemma~3.1]{GPS}, followed by Paley--Zygmund, give $\mu_R(S\ne\varnothing)\ge c_Q>0$. The preceding exceptional probabilities therefore still vanish under $\mu_R^+$. In particular, $X_R\to\infty$ and $X_R=|S|$ with probability tending to one under this law. Apply Theorem~\ref{thm:finite} with $\theta=3\delta$ to the centered normalized crossing function to obtain the face-density assertion. Lemma~\ref{lem:forest} holds for this quad and also after any defined conditioning of its pivotal law. It proves both singularity assertions. Finally, the nonempty overlap of $\mu_R$ and $\nu_R$ is at most $\mu_R(K_R=0,S\ne\varnothing)=o(1)$, whereas their empty overlap is exactly the displayed minimum.
\end{proof}

The empty-atom term need not vanish for a general quad. No quantitative rate uniform over quads is asserted in Theorem~\ref{thm:quads}.

\section{Separation under perturbations}\label{sec:perturbations}

The face statistic also gives quantitative separation in a metric on sets. For subsets of the same finite coordinate set, put
\[
 d_\triangle(A,B)=|A\mathbin\triangle B|,\qquad
 \mathcal F_R=\{B\subseteq V_R:K_R(B)=0\},\qquad
 d_\triangle(A,\mathcal F_R)=\min_{B\in\mathcal F_R}d_\triangle(A,B).
\]
This distance permits both insertions and deletions. Its minimum over $\mathcal F_R$ can always be attained by deletions alone: if $B$ is face-free, then $A\cap B$ is face-free and $d_\triangle(A,A\cap B)\le d_\triangle(A,B)$.

\begin{theorem}[Distance from face-free sets]\label{thm:distance}
For every finite connection function covered by Theorem~\ref{thm:finite}, using the corresponding domain's complete-wheel and interior-face conventions, every $m>0$, and $0\le\theta<\astar$,
\begin{equation}\label{eq:distance_tail}
 \Prob\left(d_\triangle(\cS_f,\mathcal F_R)\le\frac{\theta}{18}X_R,\ X_R\ge m\right)
 \le\exp\left\{-\frac m7D(\theta\Vert\astar)\right\}.
\end{equation}
Moreover,
\begin{equation}\label{eq:distance_mean}
 \E d_\triangle(\cS_f,\mathcal F_R)\ge\frac{\astar}{18}\E X_R.
\end{equation}
For critical triangular-site square crossings, if
\[
 W_1^\triangle(\mu_R,\nu_R)
 =\inf_{\pi\in\Pi(\mu_R,\nu_R)}\E_\pi|S\mathbin\triangle P|,
\]
where $\Pi$ denotes all couplings, then
\begin{equation}\label{eq:transport}
 W_1^\triangle(\mu_R,\nu_R)\asymp m(R).
\end{equation}
Uniformly over these couplings,
\begin{equation}\label{eq:coupling_distance}
 \Prob_\pi\left(|S\mathbin\triangle P|\ge\frac{\astar}{36}|S|\right)\longrightarrow1.
\end{equation}
\end{theorem}
\begin{proof}
Every elementary face in $A$ must lose a vertex when $A$ is changed into a face-free set. A site belongs to at most six elementary faces, hence
\begin{equation}\label{eq:delete_faces}
 d_\triangle(A,\mathcal F_R)\ge K_R(A)/6.
\end{equation}
The event in \eqref{eq:distance_tail} therefore implies $K_R(\cS_f)\le\theta X_R/3$, so Theorem~\ref{thm:finite} applies.

Let $Y_R$ count centers whose spectral supports contain an incident face. Proposition~\ref{prop:outside}, with the exterior operator equal to the identity, gives $\E Y_R\ge\astar\E X_R$. Since $Y_R\le3K_R(\cS_f)$, \eqref{eq:delete_faces} proves \eqref{eq:distance_mean}.

For a critical square, Lemma~\ref{lem:boundary} and the bulk influence comparison give $\E X_R\asymp m(R)$. Every pivotal realization belongs to $\mathcal F_R$, so every coupling has mean distance at least the left side of \eqref{eq:distance_mean}. Conversely, under any coupling, $|S\mathbin\triangle P|\le|S|+|P|$, and both expectations are comparable to $m(R)$. This proves \eqref{eq:transport}. Finally apply \eqref{eq:distance_tail} with $\theta=\astar/2$. The facts $X_R\to\infty$ in probability and $\Prob(X_R=|\cS_R|)\to1$ follow from Section~\ref{sec:critical}. The resulting exceptional event depends only on the spectral marginal, so its probability bound is uniform over all couplings.
\end{proof}

The same argument applies to the spectral law conditioned to be nonempty in a fixed Jordan quad, using Theorem~\ref{thm:quads}. Its conclusion is the analogue of \eqref{eq:coupling_distance} for either the unconditional or the nonempty-conditioned pivotal law. Thus local geometric separation cannot be removed by changing a vanishing fraction of the spectral sites. This is a statement about the number of changed sites; it does not assert a positive distance in a metric on continuum subsets.

\subsection{Stability under independent thinning}\label{sec:thinning}

Let $\cS_R^{(\rho)}$ be obtained by retaining each member of $\cS_R$ independently with probability $\rho\in(0,1]$. The retention variables are independent of the original spectral sample. Write $K_R^{(\rho)}=K_R(\cS_R^{(\rho)})$, while $X_R=|\cS_R\cap V_R^\circ|$ continues to count the original centers.

\begin{proposition}[Thinning bounds]\label{prop:thinning}
Put $a_\rho=\astar\rho^3$. For every fair site-connection function covered by Theorem~\ref{thm:finite}, every $m>0$, and $0\le\theta<a_\rho$,
\[
 \Prob\left(K_R^{(\rho)}\le\frac\theta3X_R,\ X_R\ge m\right)
 \le\exp\left\{-\frac m7D(\theta\Vert a_\rho)\right\}.
\]
In particular,
\begin{equation}\label{eq:thinning_noface}
 \Prob(K_R^{(\rho)}=0,\ X_R\ge m)\le(1-\astar\rho^3)^{m/7}.
\end{equation}
For critical triangular-site square crossings, put $m(R)=R^2\alpha_4(R)$ and let $\mu_R^{(\rho_R)}$ be the law of the thinned sample. If $\rho_R^3m(R)\to\infty$, then
\[
 \TV(\mu_R^{(\rho_R)},\lambda_R)\longrightarrow1
\]
for every sequence of laws $\lambda_R$ supported on sets with no full interior elementary face. This includes the original pivotal law and every independently thinned pivotal law. Conversely, if $\rho_R^3m(R)\to0$, then $\Prob(K_R^{(\rho_R)}>0)\to0$.
\end{proposition}
\begin{proof}
Adjoin independent Bernoulli retention variables $(\xi_v)$ to the spectral probability space. For a tested center $v$, let $y_v^{(\rho)}$ indicate that some incident face is contained in $\cS_R^{(\rho)}$, and put $x_v=\one_{\{v\in\cS_R\}}$.

An event determined by the spectral support and retention variables outside $W_v$ can first be conditioned on the exterior retention variables. Proposition~\ref{prop:outside} then gives its original local success bound with constant $\astar$. Conditional on the entire spectral support, choose the first present incident face in a fixed ordering, when one exists. Its three retention variables are unconditioned and independent, so it survives with probability $\rho^3$. Consequently, for every such exterior event $H$,
\[
 \Prob(y_v^{(\rho)}=1,H)\ge\astar\rho^3\Prob(x_v=1,H).
\]
Here $y_v^{(\rho)}\le x_v$. Within a disjoint wheel family, all previously revealed thinned markers and all other original center conditions are exterior to the current wheel. Thus Lemma~\ref{lem:bernoulli} applies with $\astar\rho^3$ in place of $\astar$. The seven-color H\"older argument and $\sum_v y_v^{(\rho)}\le3K_R^{(\rho)}$ prove the finite bounds.

The spectral lower-tail theorem of \cite[Theorem~1.1]{GPS}, the vanishing empty atom, and Lemma~\ref{lem:boundary} give
\[
 \lim_{u\downarrow0}\limsup_{R\to\infty}\Prob(X_R\le u\,m(R))=0.
\]
If $h_R=\rho_R^3m(R)\to\infty$, use $m_R=m(R)/\sqrt{h_R}$ in \eqref{eq:thinning_noface}. Then
\[
 \Prob(K_R^{(\rho_R)}=0)
 \le\Prob(X_R<m_R)+\exp\{-\astar\sqrt{h_R}/7\}\longrightarrow0.
\]
The event containing a full interior face separates this law from every $\lambda_R$ in the statement.

Finally, every elementary face contains three sites, and every site lies in at most six elementary faces. Hence $K_R(S)\le2|S|$. Independent thinning and the influence estimate give
\[
 \Prob(K_R^{(\rho_R)}>0)\le\E K_R^{(\rho_R)}
 =\rho_R^3\E K_R(\cS_R)\le2\rho_R^3\E|\cS_R|
 \le C\rho_R^3m(R).
\]
This proves the converse.
\end{proof}

Thus $m(R)^{-1/3}=R^{-1/4+o(1)}$ is the threshold scale for this triangular-face detector, in the two regimes above. The converse concerns this detector only and makes no assertion about total variation below that scale.

\subsection{Testing and relative entropy}\label{sec:testing}

The exact overlap exponent also controls statistical distinguishability after removal of the empty atom. Write
\[
 \mu_R^+=\mu_R(\,\cdot\mid S\ne\varnothing),\qquad
 \nu_R^+=\nu_R(\,\cdot\mid P\ne\varnothing).
\]
For laws $\lambda,\rho$ on a finite space, let
\[
 \operatorname{KL}(\lambda\Vert\rho)
 =\sum_A\lambda(A)\ln\frac{\lambda(A)}{\rho(A)},
\]
with value $+\infty$ if $\lambda$ charges a $\rho$-null atom.

\begin{corollary}[Nonempty overlap and information]\label{cor:testing}
For critical triangular-site square crossings,
\begin{align}
 1-\TV(\mu_R^+,\nu_R^+)&=R^{-11/12+o(1)},\label{eq:nonempty_overlap}\\
 \liminf_{R\to\infty}\frac{\operatorname{KL}(\nu_R^+\Vert\mu_R^+)}{\ln R}
 &\ge\frac{11}{12}.\label{eq:relative_entropy}
\end{align}
For all sufficiently large $R$, $\operatorname{KL}(\mu_R^+\Vert\nu_R^+)=+\infty$. With equal prior probabilities, the minimum error probability for distinguishing the two nonempty laws is $\frac12R^{-11/12+o(1)}$.
\end{corollary}
\begin{proof}
The nonempty spectral probability tends to one. The nonempty pivotal probability is bounded below: apply the bulk one- and two-point estimates to the number of pivotals in a fixed interior square and use Paley--Zygmund. Consequently both normalizing probabilities lie in $[c,1]$ for all sufficiently large $R$. Dividing the nonempty atoms by these probabilities changes their total overlap by at most a fixed factor. The nonempty overlap estimate in the proof of Theorem~\ref{thm:TVrate} gives \eqref{eq:nonempty_overlap}.

Here is the information inequality needed for \eqref{eq:relative_entropy}. Put $O=\sum_A\min\{\lambda(A),\rho(A)\}$ and $a=\sum_A\sqrt{\lambda(A)\rho(A)}$. Cauchy--Schwarz gives
\[
 a^2\le\left(\sum_A\min\{\lambda(A),\rho(A)\}\right)
 \left(\sum_A\max\{\lambda(A),\rho(A)\}\right)=O(2-O).
\]
When the relative entropy is finite, Jensen's inequality under $\lambda$ gives $\operatorname{KL}(\lambda\Vert\rho)\ge-2\ln a$; the same lower bound is automatic when it is infinite. Hence
\[
 \operatorname{KL}(\lambda\Vert\rho)\ge-\ln\bigl(O(2-O)\bigr).
\]
Apply this to $\lambda=\nu_R^+$ and $\rho=\mu_R^+$ and use \eqref{eq:nonempty_overlap}. In the reverse direction, the face event has positive spectral probability and zero pivotal probability, proving infinite relative entropy. Finally, the pointwise optimal test chooses the more likely atom; summing its error gives $O/2$.
\end{proof}

The face test alone has error of order $R^{-1/2+o(1)}$ under equal nonempty priors, by Theorem~\ref{thm:rate}. The optimal test achieves the smaller overlap scale in \eqref{eq:nonempty_overlap}; no efficient implementation of that optimal test is asserted.

\section{Scope and remaining questions}\label{sec:scope}

The two questions in \cite[Section~11]{GPS} require different inputs. The entropy argument uses spatial refinement, arm bounds, and local guards. The singularity argument additionally uses a local pattern forbidden for pivotal sets and coercive for the spectrum. The main conclusions, including their normalizations, can be summarized as follows.

\begin{center}
\small
\begin{tabularx}{\textwidth}{@{}>{\raggedright\arraybackslash}p{.285\textwidth}>{\raggedright\arraybackslash}p{.315\textwidth}>{\raggedright\arraybackslash}X@{}}
\hline
Setting & Entropy & Geometric separation\\
\hline
Critical square-bond rectangles & Both set entropies $\asymp m(R)$; matching nonroot resolution bounds & No bond singularity conclusion here\\[5pt]
Triangular-site sandwich measures, $R\le L_\varepsilon(p_*)$ & The same comparisons, uniform over parameter profiles & Finite face bounds with $a_\delta$; no asymptotic assertion without size and boundary inputs\\[5pt]
Critical triangular-site squares & Both entropies $\asymp m(R)$ & Discrete singularity, sharp rare-event and overlap rates, transport and thinning bounds\\[5pt]
Critical triangular-site Jordan quads & No new entropy assertion for arbitrary rough quads & Conditioned discrete singularity; explicit common-empty-atom correction\\
\hline
\end{tabularx}
\end{center}

Two limitations are structural. First, the triangular-face certificate does not apply directly to square-lattice bonds. A bond analogue would need a forbidden local support pattern and a corresponding local Fourier coercivity estimate, or a different separating statistic. Second, all separators used here are microscopic. They do not themselves distinguish the continuum laws: even measures whose discrete total variation equals one can have identical weak limits. For example, $\delta_{1/n}$ and $\delta_{-1/n}$ converge to the same point mass. Thus the continuum alternative in GPS Question~5 remains a separate problem.

The thinning result identifies the transition scale for the face detector, but does not determine whether the full laws remain separated below that scale. Likewise, the general-quad argument provides qualitative separation after conditioning; it does not supply rates uniform over Jordan boundaries. The near-critical finite coercivity bound isolates the additional inputs required for an asymptotic extension: nonempty spectral size must escape bounded sets, and the untested boundary region must have negligible spectral mass under the chosen normalization.

The arguments do not establish the unrestricted Fourier entropy--influence conjecture or the other open questions in \cite[Section~11]{GPS}.

\section*{Funding}
The author acknowledges support from the European Research Council under the European Union's Horizon Europe programme (grant 101041711), the Simons Foundation, Heights Labs, Convex Nexus Capital, and Israel Science Foundation grants 2258/19 and 4101/25.

\section*{Use of artificial intelligence}
During 2026, OpenAI Codex and GPT-6 were used extensively for mathematical development, literature searches, proof review, and preparation of this manuscript. The author directed the work and is responsible for the mathematical content and final version. No artificial-intelligence system is listed as an author.

\end{document}